\documentclass[11pt]{amsart}
\usepackage{mathrsfs}
\usepackage{amssymb}
\usepackage[T1]{fontenc}
\usepackage{amscd}
\usepackage{amsmath, amssymb}
\usepackage{amsthm}
\usepackage{amsfonts}
\usepackage[colorlinks,linkcolor=blue,citecolor=blue, pdfstartview=FitH]
{hyperref}
\usepackage{amscd}
\usepackage{easybmat}
\usepackage{mathrsfs}
\usepackage{amsfonts}
\usepackage{color}
\usepackage{pifont}
\usepackage{upgreek}
\usepackage{bm}
\usepackage{hyperref}
\usepackage{shorttoc}
\usepackage{amsmath,amstext,amsthm,a4,amssymb,amscd}
\usepackage[mathscr]{eucal}
\usepackage{mathrsfs}
\usepackage{epsf}

\numberwithin{equation}{section}

\def\p{\partial}
\def\o{\overline}
\def\b{\bar}
\def\mb{\mathbb}
\def\mc{\mathcal}
\def\n{\nabla}

\theoremstyle{plain}
\newtheorem{thm}{Theorem}[section]
\newtheorem{lemma}[thm]{Lemma}
\newtheorem{prop}[thm]{Proposition}
\newtheorem{cor}[thm]{Corollary}
\theoremstyle{definition}
\newtheorem{rem}[thm]{Remark}

\theoremstyle{definition}
\newtheorem{defn}[thm]{Definition}
\newcommand{\comment}[1]{}

\usepackage{fancyhdr}
\newenvironment{aligns}{\equation\aligned}{\endaligned\endequation}

\begin{document}

\title{PLURISUBHARMONICITY AND GEODESIC CONVEXITY
 OF ENERGY FUNCTION ON TEICHM\"ULLER SPACE
}

\makeatletter

\makeatother
\author{InKang Kim}
\author{Xueyuan Wan}
\author{Genkai Zhang}

\address{Inkang Kim: School of Mathematics, KIAS, Heogiro 85, Dongdaemun-gu Seoul, 130-722, Republic of Korea}
\email{inkang@kias.re.kr}

\address{Xueyuan Wan: Mathematical Sciences, Chalmers University of Technology and Mathematical Sciences, G\"oteborg University, SE-41296 G\"oteborg, Sweden}
\email{xwan@chalmers.se}

\address{Genkai Zhang: Mathematical Sciences, Chalmers University of Technology and Mathematical Sciences, G\"oteborg University, SE-41296 G\"oteborg, Sweden}
\email{genkai@chalmers.se}

\begin{abstract}
  Let $\pi:\mc{X}\to \mc{T}$ be  Teichm\"uller curve over  Teichm\"uller space $\mc{T}$,  such that the fiber $\mc{X}_z=\pi^{-1}(z)$ is exactly the Riemann surface given by the complex structure $z\in \mc{T}$. For a fixed Riemannian manifold $M$ and a continuous map $u_0: M\to \mc{X}_{z_0}$, 
let $E(z)$ denote the energy function
of the harmonic map $u(z):M\to \mc{X}_z$ homotopic to $u_0$,
$z\in \mathcal T$.
We obtain the first and the second variations of the energy function $E(z)$, and show that $\log E(z)$ is strictly plurisubharmonic  on   Teichm\"uller space, from which we give a new proof on the Steinness of Teichm\"uller space. We also obtain a precise formula on the second variation of $E^{1/2}$ if $\dim M=1$. In particular, we get the
 formula of Axelsson-Schumacher on the second variation of the geodesic length function. 
   We give also a simple and corrected proof for the theorem of Yamada, the convexity of energy function $E(t)$ along Weil-Petersson geodesics. As an application  we show that $E(t)^c$ is also strictly convex for $c>5/6$ and convex for $c=5/6$
along  Weil-Petersson geodesics. We also reprove a Kerckhoff's theorem which is a positive answer to the Nielsen realization problem. 
 \end{abstract}
\maketitle
\tableofcontents

\section*{Introduction}
Teichm\"uller space is one of the most studied objects in mathematics. It carries several natural metrics like Teichm\"uller metric, Weil-Petersson metric, Lipschitz metric etc. The Weil-Petersson metric is K\"ahler but not complete. Cheng and Yau \cite{CY} showed that there is a unique complete K\"ahler-Einstein metric on Teichm\"uller space with  constant negative scalar curvature. In this paper we shall use the Weil-Petersson metric to study  convexity
of certain energy  functionals  along
geodesics, and we study  
also the convexity with respect to the complex coordinates, 
namely the  plurisubharmonicty.

 There are many interesting and geometrically
 defined functions on Teichm\"uller space and the most
studied one might be the geodesic length function.
The geodesic length function $l(\gamma)=l(\gamma, g)$
of a closed curve $\gamma$ indeed is a well-defined
function of the hyperbolic metric $g$
corresponding to a complex structure $z\in \mathcal T$.
Kerckhoff showed in \cite{Kerockhoff}
 that for a finite  number of closed geodesics, which fill up a
 Riemann surface, the sum of the geodesic length functions provides a
 proper exhaustion of the corresponding Teichm\"uller space, and
 that the sum of length functions along any earthquake path is strictly convex. 
Wolpert \cite{Wol1, Wol2, Wol3} proved that $l(\gamma)$
 is actually convex along Weil-Petersson geodesics and plurisubharmonic, and the logarithm of a sum of geodesic length functions is also plurisubharmonic. In \cite{Wolf}, Wolf  presented  a precise formula for the second derivative of 
$l(\gamma)$  along a Weil-Petersson geodesic. By using the methods of K\"ahler geometry, Axelsson and Schumacher \cite{AS0, AS} obtained the formulas for the first and the second variation of $l(\gamma)$, and proved that its
 logarithm $\log l(\gamma)$ is strictly plurisubharmonic.

A natural generalization of the length function
is the energy function of a harmonic map. Let $\Sigma$ be a closed  surface, $M$ a Riemannian manifold of Hermitian non-positive curvature, $u_0:\Sigma\to M$ a continuous map.  Toledo \cite{Toledo} considered the energy function on Teichm\"uller space of $\Sigma$ that assigns to a complex structure 
on $\Sigma$ the energy of the harmonic map homotopic to $u_0$, and showed that this function is plurisubharmonic  on Teichm\"uller space of $\Sigma$. 
 
 Let $\mc{T}$ be Teichm\"uller space of a  surface of genus
$g\geq 2$. Let $\pi:\mc{X}\to \mc{T}$ be Teichm\"uller curve over Teichm\"uller space $\mc{T}$, namely it is the holomorphic family
of Riemann surfaces over $\mc{T}$,  the fiber $\mc{X}_z:=\pi^{-1}(z)$ being exactly the Riemann surface given by the complex structure $z\in\mc{T}$, see e.g. \cite[Section 5]{Ahlfors}.
 Let $(M^n, g)$ be a Riemannian manifold and $u_0: (M^n, g)\to (\mc{X}_z, \Phi_z)$ a continuous map, where $\Phi_z$ is the hyperbolic metric on the Riemann surface $\mc{X}_z$. For each $z\in \mc{T}$, by  \cite{ES, Hartman, Albers}, there exists a smooth harmonic map $u: (M^n, g)\to (\mc{X}_z, \Phi_z)$ homotopic to $u_0$, and it is unique unless the image of the map is  a point or a closed geodesic. By the argument in \cite[Section 1.1]{Yamada}, the following energy  
 \begin{align}\label{0.1}
 E(z)=E(u(z))=	\frac{1}{2}\int_M |du(z)|^2d\mu_g
\end{align}
is a smooth function on Teichm\"uller space (see Subsection \ref{subsec1.3}).  
In \cite{Yamada, Yamada1}, Yamada  proved the strict convexity of the energy function along the Weil-Petersson geodesics. For the case where the domain is $(\Sigma, g)$ for some hyperbolic metric
$g$, and the harmonic map $u : (\Sigma, g) \to (\mc{X}_z, \Phi_z)$ is homotopic to the identity map, the convexity has been proven by  Tromba \cite{Tromba}.
It is thus a natural question  whether the  energy function (\ref{0.1}) 
in general is plurisubharmonic on Teichm\"uller space.

Our first main theorem is 
\begin{thm}\label{main theorem}
Let $\pi:\mc{X}\to \mc{T}$ be Teichm\"uller curve over Teichm\"uller space $\mc{T}$. Let $(M^n, g)$ be a Riemannian manifold and consider the energy $E(z)$ of the harmonic map from $(M^n,g)$ to $\mc{X}_z=\pi^{-1}(z)$,  $z\in \mc{T}$. Then the logarithm of energy $\log E(z)$ is a strictly plurisubharmonic function on Teichm\"uller space. In particular, the energy function is also strictly plurisubharmonic. 	
\end{thm}
Combining with \cite[Lemma 3]{Sch1} we have the following
\begin{cor}\label{cor0.1}
	The logarithm of a sum of energy functions
	$$\log \sum_{i=1}^{N}E_i(z)$$
	is also strictly plurisubharmonic. 
\end{cor}
In the case of geodesic curves the speed 
$|du|$ is constant, so the energy function is the square of 
geodesic length function (\ref{LE}),  which implies that the logarithm of a geodesic length function is also strictly plurisubharmonic. 
\begin{cor}[\cite{Wol1, Wol2, Wol3}]\label{geodesic length}
Let $\gamma(z)$ be a smooth family of closed geodesic curves over  Teichm\"uller space. Then both the length function $\ell(\gamma(z))$ and the logarithm of length function $\log \ell(\gamma(z))$ are strictly plurisubharmonic.	 In particular, the geodesic length function is strictly convex along Weil-Petersson geodesics.
	
\end{cor}

In \cite{Wol1}, the geodesic length function of a family of curves that fill up the surface
is proved to be proper and plurisubharmonic, then Wolpert \cite[Section 6]{Wol1} gave a new proof on  Steinness of Teichm\"uller space \cite{Bers}. In \cite[Theorem 6.1.1]{Tromba0}, Tromba also reproved this result using Dirichlet's energy, which is a function on Teichm\"uller space of the initial manifold.  For the properness of energy function,  Wolf \cite{Wolf1} proved that the energy function is proper if the domain manifold is a hyperbolic surface $(\Sigma, g)$ with the harmonic map homotopic to the identity. For a general Riemannian manifold $M$, Yamada \cite[Proposition 3.2.1]{Yamada} showed the properness of energy function when $(u_0)_*: \pi_1(M)\to \pi_1(\mc{X}_{z_0})$ is surjective. 
Combining with  Theorem \ref{main theorem}, this shows that $\mc{T}$ is Stein. 
\begin{cor}\label{stein}
If $(u_0)_*: \pi_1(M)\to \pi_1(\mc{X}_{z_0})$ is surjective, then the energy function $E(z)$ is proper and strictly plurisubharmonic. In particular, Teichm\"uller space $\mc{T}$ is a complex Stein manifold. 
\end{cor}

We explain briefly our method to prove Theorem \ref{main theorem}.

Let $u: (M^n, g)\to (\mc{X}_z,\Phi)$ be a smooth map, then $du$ is the section of bundle $T^*M\otimes u^*T_{\mb{C}}\mc{X}_z$, for which there is an induced metric $g^*\otimes \Phi$ from $(M^n,g)$ and $(\mc{X}_z,\Phi)$. Here $T_{\mb{C}}\mc{X}_z=T\mc{X}_z\oplus \o{T\mc{X}_z}$ denotes the complex tangent bundle, and $T\mc{X}_z$ denotes the holomorphic tangent bundle of $\mc{X}_z$.
 Let $\{x^i\}$ denote a local coordinate system near a point $p$ in $M$,  and $\{v\}$ denote the holomorphic coordinates of Riemann surface $\mc{X}_z$. Let $z=\{z^{\alpha}\}$ denote the holomorphic coordinates of Teichm\"uller space $\mc{T}$,  the following
tensor will play a crucial role in our computation,
\begin{align}\label{Aalpha}
A_{\alpha}=A_{\alpha\b{v}\b{v}}\o{u^v_i}\phi^{v\b{v}}dx^i\otimes \frac{\p}{\p v}\in A^1(M, u^*T\mc{X}_z);
\end{align}
see Subsection \ref{subsec1.1} 
for the precise definition.
\begin{thm}\label{thm0.3}
The first variation of the energy function $E(z)$ (\ref{0.1}) is given by 
\begin{align}
\frac{\p}{\p z^{\alpha}}E(z)=\langle A_{\alpha}, du\rangle.
\end{align}
\end{thm}
Let $\Delta=\n\n^*+\n^*\n$ be the Hodge-Laplace operator on $A^{\ell}(M, u^*T\mc{X}_z)$ (see Subsection \ref{subsec1.2}), and set 
\begin{align*}
\mc{L}=\Delta+\frac{1}{2}|du|^2,\quad
	\mc{G}=g^{ij}\phi_{v\b{v}}u^v_i u^v_j \frac{\p}{\p v}\otimes d\b{v}\in \text{Hom}(u^*\o{T\mc{X}_z},u^*T\mc{X}_z),	
\end{align*}
and $c(\phi)_{\alpha\b{\beta}}:=\phi_{\alpha\b{\beta}}-\phi_{\alpha\b{v}}\phi_{v\b{\beta}}\phi^{v\b{v}}$ (see Lemma \ref{lemma0}). Then
\begin{thm}\label{thm0.4}
	The second variation of the energy (\ref{0.1}) is given by
	\begin{align}\label{0.2}
	\frac{\p^2}{\p z^{\alpha}\p\b{z}^{\beta}}E(z)=\frac{1}{2}\int_M c(\phi)_{\alpha\b{\beta}}|du|^2d\mu_g+\langle(Id-\n\left(\mc{L}-\mc{G}\mc{L}^{-1}\o{\mc{G}}\right)^{-1}\n^*)A_{\alpha},A_{\beta}\rangle.	
	\end{align}
\end{thm}
It is a well-known fact that
 in the RHS of (\ref{0.2}) is positive; see
\cite[Theorem 1]{Sch}. We shall  show that the second term  in the RHS of (\ref{0.2}) is  non-negative, proving
thus 
 the plurisubharmonicity. 
The easiest case is when  $\dim M=1$, i.e, $u$ is a geodesic curve.
Then  $\n^2=0$ and we get
    \begin{prop}\label{prop2}
If $\dim M=1$, then 	
\begin{align*}
	\frac{\p^2 E^{1/2}}{\p z^{\alpha}\p\b{z}^{\beta}} 
 &=	\frac{1}{2}\frac{1}{E^{1/2}}\left(\int_M (\Box+1)^{-1}(A_{\alpha},A_{\beta}) d\mu_g+\langle\frac{1}{2}|du|^2(|du|^2+\Delta)^{-1}A_{\alpha}, A_{\beta}\rangle\right),
\end{align*}
where $\Box=-\phi^{v\b{v}}\p_v\p_{\b{v}}$ and $(A_{\alpha},A_{\beta})=A_{\alpha\b{v}}^v\o{A_{\beta\b{v}}^v}(\frac{1}{2}|du|^2)$ is a smooth function on $(z,v)=(z,u(z,x))$.
If we take the arc-length parametrization at $z=z_0$, i.e. $\frac{1}{2}|du|^2(z_0)=1$,  then the first and the second variations  of geodesic length function are given by
\begin{align*}
	\frac{\p\ell(z)}{\p z^{\alpha}}|_{z=z_0}=\frac{1}{2}\langle A_{\alpha},du\rangle
\end{align*}
and
\footnote{We note that in  \cite[Theorem 6.2,(38)]{AS}, 
there is an extra term
$\frac{1}{4\ell(\gamma_s)}\int_{\gamma_s}A_i\cdot\int_{\gamma_s}A_{\b{j}}$
appearing, this is due to a minor miscomputation; see Remark \ref{rem} below .}
 \begin{align*}
 \frac{\p^2\ell(z)}{\p z^{\alpha}\p\b{z}^{\beta}}|_{z=z_0}=\frac{1}{2}\left(\int_M (\Box+1)^{-1}(A_{\alpha},A_{\beta})d\mu_g+\langle(2+\Delta)^{-1}A_{\alpha}, A_{\beta}\rangle\right).
 \end{align*} 
\end{prop}
For higher dimensional $M$ $\n^2\not\equiv 0$ 
generally (see e.g. \cite[Page 15]{Xin}), and we
shall treat the
second term in more details.
For notational convinience
we may assume without loss of generality
that the base manifold $\mc{T}$ is one dimensional,
with the indices $\alpha,\beta$ being  replaced by $z$, $A:=A_{\alpha}$. 
 A major ingredient of our proof is the following decomposition
\begin{multline}
Id-\n\left(\mc{L}-\mc{G}\mc{L}^{-1}\o{\mc{G}}\right)^{-1}\n^*=(\Delta^{-1}\Delta-\n\Delta^{-1}\n^*)\\+(\n\Delta^{-1}\n^*-\n\left(\mc{L}-\mc{G}\mc{L}^{-1}\o{\mc{G}}\right)^{-1}\n^*)+\mb{H}, 
 \end{multline}
where 
$\mb{H}$ is the orthogonal projection 
onto harmonic forms.  By Lemmas \ref{lemma9} and \ref{lemma7},
 both the operators $(\Delta^{-1}\Delta-\n\Delta^{-1}\n^*)$ and $(\n\Delta^{-1}\n^*-\n\left(\mc{L}-\mc{G}\mc{L}^{-1}\o{\mc{G}}\right)^{-1}\n^*)$ are non-negative when acting on  $A^{1}(M,u^*T\mc{X}_z)$. Thus 
\begin{aligns}\label{0.3}
\frac{\p^2}{\p z\p\b{z}}E(z)&\geq \frac{1}{2}\int_M c(\phi)_{z\b{z}}|du|^2d\mu_g+\|\mb{H}(A)\|^2.
\end{aligns}
Note that $du$ is  harmonic (see Proposition \ref{prop1}) and using Theorem \ref{thm0.3}, we obtain a lower bound for $\|\mb{H}(A)\|^2$, namely
\begin{aligns}\label{0.4}
\|\mb{H}(A)\|^2&\geq \frac{1}{E}|\langle A, du\rangle|^2=	\frac{1}{E}\frac{\p E}{\p z}\frac{\p E}{\p\b{z}},
\end{aligns}
Combining (\ref{0.3}) with (\ref{0.4}) yields
\begin{aligns}
\frac{\p^2}{\p z\p\b{z}}\log E(z) 
&\geq \frac{1}{\|du\|^2}\int_M c(\phi)_{z\b{z}}|du|^2d\mu_g>0,
\end{aligns}
which proves Theorem \ref{main theorem}. 

Next we  give a simple proof for a theorem of Yamada on
 the convexity of the energy function along Weil-Petersson geodesic\footnote{His proof has a gap. On \cite[Page 62]{Yamada}, the Schwarz inequality is  used
as $ab\leq \frac{1}{4}(a^2+b^2)$ by mistake. 
It seems
to us that with the correct use of the Schwarz inequality the method
there can  not lead to a proof of the convexity.
We thank Yamada for correspondences on this matter.}.
\begin{thm}[{\cite[Theorem 3.1.1]{Yamada}}]\label{thm0.5}
 The energy function $E: \mc{T}\to \mb{R}$,  (\ref{0.1}), is strictly convex along any Weil-Petersson geodesic in $\mc{T}$.	
\end{thm}

Our major method is simply to use
the splitting of the
the tensor  $\nabla W$, $W=u'(0)$, for a family $u(t): M\to (\Sigma, \Phi_t)$
of harmonic maps along a geodesic under
the decomposition of $T_{u(0, x)}\Sigma =T^{(1, 0)}
+T^{(0, 1)}$, along with the following expansion for hyperbolic metrics $\Phi_t$ along the Weil-Petersson geodesic in $\mc{T}$,
\begin{multline}\label{WPG-0}
	\Phi_t=\phi_0dvd\b{v}+t(q dv^2+\o{q}d\b{v}^2)\\
	+t^2/2\left(\frac{2|q|^2}{\phi^2_0}-2(\Delta-2)^{-1}\frac{2|q|^2}{\phi^2_0}\right)\phi_0dvd\b{v}+O(t^4),
\end{multline}
(see \cite[(3.4)]{Wolf}).  Here $qdv^2$ is a holomorphic quadratic form, $\phi_0dvd\b{v}$ is a hyperbolic metric. It is also noticed in \cite{Toledo} 
that the splitting above is critical in proving the
plurisubharmonicity for the energy of harmonic maps $u(z): \mathcal X_z\to M$.

As a corollary, we prove that  
\begin{cor}\label{corY1}
The function $E(t)^{c}, c>5/6$ (resp. $c=5/6$) is strictly convex (resp. convex) along a Weil-Petersson geodesic. 	
\end{cor}
Another corollary of Theorem \ref{thm0.5} is a positive answer to the Nielsen realization problem, which was proved by Kerckhoff \cite{Kerockhoff}.
\begin{cor}[{\cite[Theorem 5]{Kerockhoff}}]\label{corY2}
Any finite subgroup of the mapping class group of a closed surface $\Sigma$ of genus greater than 1 can be realized as an isometry subgroup
of some hyperbolic metric on $\Sigma$.
\end{cor}
This article is organized as follows. In Section \ref{sec1}, we fix notation and recall
some basic facts on Teichm\"uller curve, Hodge-Laplace operator and Harmonic maps.
In Section \ref{sec2}, we compute the first and second variations of the energy function (\ref{0.1}) and prove Theorem \ref{thm0.3}, \ref{thm0.4} and Proposition \ref{prop2}. In Subsection \ref{subsec2.3}, we will show the strict plurisubharmonicity of logarithmic energy and prove Theorem \ref{main theorem}, Corollary \ref{cor0.1}, \ref{geodesic length}, \ref{stein}. In the last section, we give a simple proof on  convexity of the energy function along Weil-Petersson geodesic, i.e. Theorem \ref{thm0.5}, and then prove Corollary \ref{corY1}, \ref{corY2}.

\vspace{5mm}
{\bf Acknowledgment:} This work was begun when the second and the third authors were visiting the first author at Korea Institute for Advanced Study (KIAS) during June 2018-July 2018. We thank  KIAS for its support and 
for providing excellent working
environment.

\section{Preliminaries}\label{sec1}
\subsection{Teichm\"uller curve}\label{subsec1.1}
The results in this subsection are well-known.
Let $\mc{T}$ be Teichm\"uller space of a fixed surface of genus
$g\geq 2$. Let $\pi:\mc{X}\to \mc{T}$ be Teichm\"uller curve over Teichm\"uller space $\mc{T}$, namely the holomorphic family
of Riemann surfaces over $\mc{T}$,  the fiber $\mc{X}_z:=\pi^{-1}(z)$ being exactly the Riemann surface given by the complex structure $z\in\mc{T}$; see e.g. \cite[Section 5]{Ahlfors}. Denote by
$$(z;v)=(z^1,\cdots, z^m; v)$$
the local holomorphic coordinates of $\mc{X}$ with $\pi(z,v)=z$, where
$z=(z^1,\cdots z^m)$ denotes the local coordinates of $\mc{T}$ and $v$
denotes the local coordinates of Riemann surface $\mc{X}_z$,
$m=3g-3=\dim_{\mathbb C} \mc{T}$.  Let $K_{\mc{X}/\mc{T}}$ denote the  relative canonical line bundle over $\mc{X}$, when restricts to each fiber $K_{\mc{X}/\mc{T}}|_{\mc{X}_z}=K_{\mc{X}_z}$. The fibers $\mc{X}_z$ are equipped with hyperbolic metric 
\begin{align*}
\omega_{\mc{X}_z}=\sqrt{-1}\phi_{v\b{v}}dv\wedge d\b{v}
\end{align*}
depending smoothly on the parameter $z$ and having negative constant curvature $-1$, namely, 
\begin{align}\label{1.1}
	\p_v\p_{\b{v}}\log\phi_{v\b{v}}=\phi_{v\b{v}},
\end{align}
where $\phi_{v\b{v}}:=\p_v\p_{\b{v}}\phi$.
From (\ref{1.1}), up to a scaling
function on $\mc{T}$  a metric (weight) $\phi$ on $K_{\mc{X}/\mc{T}}$ can
be chosen such that 
\begin{align}\label{1.2}
	e^{\phi}=\phi_{v\b{v}}.
\end{align}
For convenience,  we denote
$$\phi_{\alpha}:=\frac{\p \phi}{\p z^{\alpha}},\quad \phi_{\b{\beta}}:=\frac{\p \phi}{\p \b{z}^{\beta}},
\quad \phi_{v}:=\frac{\p \phi}{\p v},\quad \phi_{\b{v}}:=\frac{\p \phi}{\p \b{v}},$$
where $1\leq \alpha,\beta\leq m$.
Denote 
$\omega=\sqrt{-1}\p\b{\p}\phi$. With respect to the $(1,1)$-form $\omega$,
we have a canonical horizontal-vertical decomposition of $T\mc{X}$, $T\mc{X}=\mc{H}\oplus \mc{V}$, where
\begin{align*}
\mc{H}=\text{Span}\left\{\frac{\delta}{\delta z^{\alpha}}=\frac{\p}{\p z^{\alpha}}+a^v_\alpha \frac{\p}{\p v}, 1\leq \alpha\leq m\right\},\quad \mc{V}=\text{Span}\left\{\frac{\p}{\p v}\right\},
\end{align*}
where \begin{align}\label{1.3}
a^v_{\alpha}=-\phi_{\alpha\b{v}}\phi^{v\b{v}},
\end{align}
and $\phi^{v\b{v}}=(\phi_{v\b{v}})^{-1}$.
By duality,  $T^*\mc{X}=\mc{H}^*\oplus \mc{V}^*$, where 
\begin{align*}
\mc{H}^*=\text{Span}\left\{dz^{\alpha}, 1\leq \alpha\leq m\right\},\quad \mc{V}^*=\text{Span}\left\{\delta v=dv-a^v_{\alpha}dz^{\alpha}\right\}.
\end{align*}
Moreover, the differential operators
\begin{align}\label{HV}
\p^V=\frac{\p}{\p v}\otimes \delta v,\quad \p^H=\frac{\delta}{\delta z^{\alpha}}\otimes dz^{\alpha}, \quad \b{\p}^V=\frac{\p}{\p\bar v}\otimes \delta \bar v, \quad \b{\p}^H=\frac{\delta}{\delta \bar z^{\alpha}}\otimes d\bar z^{\alpha}
\end{align}
are well-defined. The following lemma can be proved by direct computations.
\begin{lemma}[{\cite[Lemma 1.1]{Wan1}}]\label{lemma0}
We have the following decomposition of the K\"a{}hler form $\omega$:
\begin{align*}
\omega=\sqrt{-1}\p\b{\p}\phi=c(\phi)+\sqrt{-1}\phi_{v\b{v}}\delta v\wedge \delta\b{v},	
\end{align*}
where $c(\phi)=\sqrt{-1}c(\phi)_{\alpha\b{\beta}}dz^{\alpha}\wedge d\b{z}^{\beta}$, $c(\phi)_{\alpha\b{\beta}}=\phi_{\alpha\b{\beta}}-\phi^{v\b{v}}\phi_{\alpha\b{v}}\phi_{v\b{\beta}}$. 
\end{lemma}
 We  consider the following tensor
 \begin{align}
   \label{A-pre}
\b{\p}^V\frac{\delta}{\delta z^{\alpha}}=(\p_{\b{v}}a^v_{\alpha}) \frac{\p}{\p v}\otimes \delta\b{v}\in A^0(\mc{X}, \text{End}(\mc{V})).
\end{align}
By Lemma \ref{lemma1} (iii) we see 
that its restriction to each fiber
is a harmonic element
representating  the Kodaira-Spencer class $\rho(\frac{\p}{\p z^{\alpha}})$,
$\rho: T_z\mc{T}\to H^1(\mc{X}_z,T_{\mc{X}_z})$ being
 the Kodaira-Spencer map.
We denote its component and its dual
with respect to the metric 
$\sqrt{-1}\phi_{v\b{v}}\delta v\wedge \delta \b{v}$ as  
\begin{align}\label{1.4}
A^v_{\alpha\b{v}}=\p_{\b{v}}a^v_\alpha=\p_{\b{v}}(-\phi^{v\b{v}}\phi_{\alpha\b{v}}),\quad A_{\alpha\b{v}\b{v}}=A^v_{\alpha\b{v}}\phi_{v\b{v}}.	 
\end{align}
Note that $$(\mc{V},\sqrt{-1}\phi_{v\b{v}}\delta v\wedge \delta \b{v})$$ is a Hermitian vector bundle over $\mc{X}$ as well as $\text{End}(\mc{V})=\mc{V}\otimes\mc{V}^*$.
We denote by 
$\n_v, \n_{\b{v}}$ the covariant derivatives  along the directions 
$\p/\p v, \p/\p\b{v}$, respectively. 
For convenience, we also denote by ${}_{;v}, {}_{;\b{v}}$ the
covariant derivatives $\n_v, \n_{\b{v}}$.

From \cite[Proposition 2.1, Lemma 2.2, Lemma 5.2]{AS} and \cite[Theorem 1, Proposition 3]{Sch}, we have the following lemma; 
we provide for the first four identities.
\begin{lemma}[\cite{AS,Sch}]\label{lemma1}
The following identities
 hold:
\begin{enumerate}
  \item[(i)] $a^v_{\alpha_;vv}=-\phi_{\alpha v}$;
  \item[(ii)] $A_{\alpha\b{v}\b{v}}=-\n_{\b{v}}\phi_{\alpha\b{v}}=-\phi_{\alpha\b{v};\b{v}}$, $a^v_{\alpha;\bar v}=A_{\alpha \bar v\bar v}\phi^{v\bar v}$;
  \item[(iii)] $\p_vA_{\alpha\b{v}\b{v}}=A_{\alpha\b{v}\b{v};v}=0$;
  \item[(iv)] $\frac{\p}{\p \b{z}^{\beta}}A_{\alpha\b{v}\b{v}}=-c(\phi)_{\alpha\b{\beta};\b{v}\b{v}}-2A_{\alpha\b{v}\b{v}}\o{a^v_{\beta;v}}-A_{\alpha\b{v}\b{v};\b{v}}\o{a^v_\beta}$;
  \item[(v)] $(\Box+1)c(\phi)_{\alpha\b{\beta}}=A^{v}_{\alpha\b{v}}A^{\b{v}}_{\b{\beta}v}$ where $\Box=-\phi^{v\b{v}}\p_v\p_{\b{v}}$ and   $c(\phi)_{\alpha\b{\beta}}=\phi_{\alpha\b{\beta}}-\phi^{v\b{v}}\phi_{\alpha\b{v}}\phi_{v\b{\beta}}$;
  \item[(vi)] $c(\phi)\geq 0$, and $c(\phi)>0$ if the family is not infinitesimally trivial. In particular, for Teichm\"uller curve $\pi:\mc{X}\to\mc{T}$, one has 
  $$c(\phi)\geq P_1(d(\mc{X}_z))\pi^*\omega^{WP}>0$$
  along the horizontal directions,
  where $P_1(d(\mc{X}_z))$ is a strictly positive function  depending on the diameter $d(\mc{X}_z)$, and $\omega^{WP}$ is the Weil-Petersson metric on Teichm\"uller space $\mc{T}$.  \\
\end{enumerate}
\end{lemma}
\begin{proof}
	(i) By (\ref{1.3}) and $\n_v\phi^{v\b{v}}=0$ one has
	\begin{align*}
		a^v_{\alpha;vv}=(-\phi_{\alpha\b{v}}\phi^{v\b{v}})_{;vv}=(-\phi_{\alpha v\b{v}}\phi^{v\b{v}})_{;v}=(-\p_{\alpha}\log\phi_{v\b{v}})_{;v}=-\phi_{\alpha v},
	\end{align*}
where the last equality follows from (\ref{1.2}).\\
(ii) By (\ref{1.4}),
\begin{align*}
A_{\alpha\b{v}\b{v}}&=A^{v}_{\alpha\b{v}}\phi_{v\b{v}}=\p_{\b{v}}(-\phi^{\b{v}v}\phi_{\alpha\b{v}})\phi_{v\b{v}}\\
&=-(\p_{\b{v}}\phi_{\alpha\b{v}}-\phi_{\alpha\b{v}}\p_{\b{v}}\log\phi_{v\b{v}})\\
&=-\n_{\b{v}}\phi_{\alpha\b{v}}=-\phi_{\alpha\b{v};\b{v}}. 	
\end{align*}
(iii) By (\ref{1.3}) and a direct computation
\begin{align*}
\partial_v A_{\alpha \bar v\bar v}=A_{\alpha\b{v}\b{v};v}&=\p_v(\p_{\b{v}}(-\phi_{\alpha\b{v}}\phi^{\b{v}v})\phi_{v\b{v}})\\
&=\p_v(-\phi_{\alpha\b{v}\b{v}}+\phi_{\alpha\b{v}}\p_{\b{v}}\log \phi_{v\b{v}})\\
&=-\phi_{v\b{v}\alpha\bar v}+\phi_{\alpha v\b{v}}\p_{\b{v}}\log\phi_{v\b{v}}+\phi_{\alpha \b{v}}\p_v\p_{\b{v}}\log\phi_{v\b{v}}\\
&=-(e^{\phi})_{\alpha\b{v}}+(e^{\phi})_\alpha \p_{\b{v}}\phi+\phi_{\alpha\b{v}}\phi_{v\b{v}}\\
&=-(e^{\phi})_{\alpha\b{v}}+e^{\phi}\phi_\alpha\phi_{\b{v}}+\phi_{\alpha\b{v}}e^{\phi}=0.
\end{align*}
(iv) Similar calculations give
\begin{align*}
\frac{\p}{\p \b{z}^{\beta}}A_{\alpha \b{v}\b{v}}&=\frac{\p}{\p \b{z}^{\beta}}(-\p_{\b{v}}\phi_{\alpha\b{v}}+\phi_{\alpha\b{v}}(\p_{\b{v}}\log\phi_{v\b{v}}))	\\
&=-\phi_{\alpha\b{\beta}\b{v};\b{v}}+\phi_{\alpha\b{v}}\p_{\b{\beta}}\p_{\b{v}}\log\phi_{v\b{v}}\\
&=-(c(\phi)_{\alpha\b{\beta}}+a^v_\alpha\o{a^v_\beta}\phi_{v\b{v}})_{;\b{v}\b{v}}+\phi_{\alpha\b{v}}(\phi_{\b{\beta}v}\phi^{v\b{v}})_{;\b{v}\b{v}}\\
&=-c(\phi)_{\alpha\b{\beta};\b{v}\b{v}}-(a^v_\alpha\o{a^v_\beta}\phi_{v\b{v}})_{;\b{v}\b{v}}-\phi_{\alpha\b{v}}a^{\b{v}}_{\b{\beta};\b{v}\b{v}}\\
&=-c(\phi)_{\alpha\b{\beta};\b{v}\b{v}}-2A_{\alpha\b{v}\b{v}}\o{a^v_{\beta;v}}-A_{\alpha\b{v}\b{v};\b{v}}\o{a^v_\beta}.
\end{align*}
For $(v)$ and $(vi)$, one can refer to \cite[Theorem 1, Proposition 3]{Sch}.
\end{proof}

\subsection{Hodge-Laplacian}\label{subsec1.2}
Let 
\begin{align}
\Phi=\phi_{v\b{v}}(dv\otimes d\b{v}+d\b{v}\otimes dv)	
\end{align}
denote the Riemannian metric on $\mc{X}_z$ associated to the fundamental form $\omega|_{\mc{X}_z}=\sqrt{-1}\phi_{v\b{v}}dv\wedge d\b{v}$.
Let $T\mc{X}_z$ denote the holomorphic tangent bundle of $\mc{X}_z$ and $T_{\mb{C}}\mc{X}_z=T\mc{X}_z\oplus\o{T\mc{X}_z}$ denote the complex tangent bundle.
 For any smooth map $u$  from a Riemannian manifold $(M^n, g)$ to $(\mc{X}_z, \Phi)$ and for any $\ell\geq 0$, there is a natural connection on $\wedge^{\ell}T^*M\otimes u^*T_{\mb{C}}\mc{X}_z$ induced from the Levi-Civita connections of $(M^n, g)$ and $(\mc{X}_z, \Phi)$, and we denote by $\n_i$ (or ${}_{;i}$)  the covariant derivatives along the vector $\frac{\p}{\p x^i}$. For example, for the tensor $\Psi=\Psi^{j_1\cdots j_s v^{m_1}\b{v}^{m_2}}_{k_1\cdots k_l v^{n_1}\b{v}^{n_2}}dx^{k_1}\otimes\cdots\otimes dx^{k_l}\otimes \frac{\p}{\p x^{j_1}}\otimes \cdots\otimes\frac{\p}{\p x^{j_s}}\otimes (dv)^{n_1-m_1}\otimes (d\b{v})^{n_2-m_2}$, where $v^{n_1}$ denotes $\underbrace{v\cdots v}_{n_1}$ and  $(dv)^{-1}:=\p/\p v$, one has
 \begin{multline}\label{cov1}
 \n_i \Psi=(\n_i\Psi^{j_1\cdots j_s v^{m_1}\b{v}^{m_2}}_{k_1\cdots k_l v^{n_1}\b{v}^{n_2}})dx^{k_1}\otimes\cdots\otimes dx^{k_l}\otimes \frac{\p}{\p x^{j_1}}\otimes \cdots\otimes\frac{\p}{\p x^{j_s}}\\\otimes (dv)^{n_1-m_1}\otimes (d\b{v})^{n_2-m_2},	 
 \end{multline}
where 
\begin{multline}\label{cov2}
	\n_i\Psi^{j_1\cdots j_s v^{m_1}\b{v}^{m_2}}_{k_1\cdots k_l v^{n_1}\b{v}^{n_2}} =\frac{\p}{\p x^i}\Psi^{j_1\cdots j_s v^{m_1}\b{v}^{m_2}}_{k_1\cdots k_l v^{n_1}\b{v}^{n_2}}+\sum_{t=1}^s\sum_{p=1}^n\Gamma^{j_t}_{ip}\Psi^{j_1\cdots p\cdots j_s v^{m_1}\b{v}^{m_2}}_{k_1\cdots k_l v^{n_1}\b{v}^{n_2}}\\-\sum_{t=1}^l\sum_{p=1}^n \Gamma^p_{ik_t}\Psi^{j_1\cdots j_s v^{m_1}\b{v}^{m_2}}_{k_1\cdots p\cdots k_l v^{n_1}\b{v}^{n_2}}
	+\left((m_1-n_1)\phi_v u^v_i+(m_2-n_2)\phi_{\b{v}}\o{u^v_i}\right)\Psi^{j_1\cdots j_s v^{m_1}\b{v}^{m_2}}_{k_1\cdots k_l v^{n_1}\b{v}^{n_2}}.
\end{multline}
Here
$$\Gamma^{k}_{ij}=\frac{1}{2}g^{kl}\left(\p_jg_{il}+ \p_i g_{jl}- \p_lg_{ij}\right)$$
denote the Christoffel symbols.
We define also 
\begin{aligns}\label{nabla}
\n: \quad  & 
A^{\ell}(M, u^*T_{\mb{C}}\mc{X}_z) \to A^{\ell+1}(M, u^*T_{\mb{C}}\mc{X}_z)\\
& \quad A  \quad\quad\quad\,\mapsto \quad\quad\n A=dx^i\wedge (\n_{i}A).
\end{aligns}
(It is sometimes denoted by $d^\nabla$ to indicate the anti-symmetrization
as one may also define $\nabla$ from
$A^{\ell}(M, u^*T_{\mb{C}}\mc{X}_z)$ to 
$A^{\ell}(M, u^*T_{\mb{C}}\mc{X}_z\otimes T^\ast M)$; see 
Remark \ref{rem2.5}).

Let $(\cdot,\cdot)$ denote the pointwise inner product on $ A^{\ell}(M, u^*T_{\mb{C}}\mc{X}_z) $ induced from  $(M^n,g)$ and $(\mc{X}_z,\Phi)$; for example, for the space $A^1(M,u^*T_{\mb{C}}\mc{X}_z)$, the pointwise inner product is given by
\begin{multline}\left((f_1)_i dx^i\otimes \frac{\partial}{\partial v}+(f_2)_i dx^i\otimes \frac{\partial}{\partial \bar v}, (g_1)_j dx^j\otimes \frac{\partial}{\partial v}+(g_2)_j dx^j\otimes \frac{\partial}{\partial\bar v}\right)\\
=(f_1)_i\overline{(g_1)_j} g^{ij}\phi_{v\bar v}+(f_2)_i\overline{(g_2)_j} g^{ij}\phi_{v\bar v}.
\end{multline}
Define then the corresponding $L^2$-inner product by
\begin{align}\label{inner}
\langle \cdot,\cdot\rangle=\int_M (\cdot,\cdot)d\mu_g.	
\end{align}
  Let $\n^*$ be the adjoint operator of $\n$ with respect to the $L^2$-inner product (\ref{inner}) and define the Hodge-Laplace operator as follows:
  \begin{align}
  \Delta=\n^*\n+\n\n^*,
  \end{align}
see e.g. \cite[(1.38)]{Xin}.  By \cite[Proposition 1.32]{Xin}, $\Delta$ is a self-adjoint and semi-positive elliptic operator. Let
\begin{align}
H=\text{Ker}\Delta=\text{Ker}\n\cap \text{Ker}\n^*
\end{align}
denote the space of harmonic forms.
\begin{lemma}\label{lemma8}
It holds the following identity:
\begin{align}
\text{Id}=\Delta^{-1}\Delta+\mb{H}	
\end{align}
	when acting on the elements of $A^{\ell}(M,u^*T_{\mb{C}}\mc{X}_z)=C^{\infty}(M,\wedge^{\ell} T^*M\otimes u^*T_{\mb{C}}\mc{X}_z)$. Here $\Delta^{-1}: \text{Im}\Delta\to \text{Im}\Delta$ denotes the inverse operator of $\Delta$, and $\mb{H}$ denotes the harmonic projection from $A^{\ell}(M,u^*T_{\mb{C}}\mc{X}_z)$ to $H$. 
\end{lemma}
\begin{proof}
  From \cite[Corollary 2.4]{Dem}, considered as a operator on $A^{\ell}(M,u^*T_{\mb{C}}\mc{X}_z)=C^{\infty}(M,\wedge^{\ell} T^*M\otimes u^*T_{\mb{C}}\mc{X}_z)$,
  there is the following orthogonal decomposition:
\begin{align}\label{1.77}
A^{\ell}(M,u^*T_{\mb{C}}\mc{X}_z)=\text{Im}\Delta\oplus \text{Ker}\Delta.	
\end{align}
For any $\alpha\in A^{\ell}(M,u^*T_{\mb{C}}\mc{X}_z)$, by (\ref{1.77})
and $\Delta\mb{H}=0$ it holds
\begin{align*}
	\alpha &=\Delta\alpha_1+\mb{H}(\alpha)=\Delta^{-1}\Delta\Delta\alpha_1+\mb{H}(\alpha)\\
	&=\Delta^{-1}\Delta\left(\Delta\alpha_1+\mb{H}(\alpha)\right)+\mb{H}(\alpha)\\
	&=\Delta^{-1}\Delta(\alpha)+\mb{H}(\alpha),
\end{align*}
which completes the proof. 	
\end{proof}

\begin{lemma}\label{lemma9}
 For any $s\in A^{1}(M,u^*T_{\mb{C}}\mc{X}_z)$, we have
\begin{align}\label{1.11}
	\langle (\Delta^{-1}\Delta-\n\Delta^{-1}\n^*)s,s\rangle\geq 0.
\end{align}
\end{lemma}
\begin{proof}
For any $s\in A^{1}(M,u^*T_{\mb{C}}\mc{X}_z)$, $\nabla^* s$ is smooth and  orthogonal to $\text{Ker} \Delta$. Hence it is in the image of $\Delta$.
The same is true for $\nabla s$.
Then we have 
\begin{align*}
(\n\Delta^{-1}\n^*)^2s &=	\n\Delta^{-1}(\n^*\n)\Delta^{-1}\n^*s\\
&=\n\Delta^{-1}(\n^*\n+\n\n^*)\Delta^{-1}\n^*s\\
&=\n\Delta^{-1}\n^*s,
\end{align*}
where the second equality holds since $\n^*(\Delta^{-1})\n^*s=0$. This implies that $\n\Delta^{-1}\n^*$ is identity when acting on $\text{Im}(\n\Delta^{-1}\n^*)$.
For any $s'\in \o{\text{Im}(\n\Delta^{-1}\n^*)}$, there exists a sequence $\{s'_n\}\in \text{Im}(\n\Delta^{-1}\n^*)$ such that $s'=\lim_{n\to\infty} s_n'$, then
\begin{aligns}
\langle \n\Delta^{-1}\n^*s, s'\rangle &=\langle \n\Delta^{-1}\n^*s, \lim_{n\to\infty}s'_n\rangle
=\lim_{n\to\infty}\langle \n\Delta^{-1}\n^*s, s'_n\rangle\\
&=\lim_{n\to\infty}\langle s,  \n\Delta^{-1}\n^*s'_n\rangle
=\lim_{n\to\infty}\langle s,  s'_n\rangle\\
&=	\langle s,  \lim_{n\to\infty}s'_n\rangle=\langle s, s'\rangle\\
&=\langle P_{\o{\text{Im}(\n\Delta^{-1}\n^*)}} s, s'\rangle,
\end{aligns}
where $P_{\o{\text{Im}(\n\Delta^{-1}\n^*)}}$ is the orthogonal projection from $A^{1}(M,u^*T_{\mb{C}}\mc{X}_z)$ to  $\o{\text{Im}(\n\Delta^{-1}\n^*)}$.
It follows that $$\n\Delta^{-1}\n^*=P_{\o{\text{Im}(\n\Delta^{-1}\n^*)}}.$$ Note that $\o{\text{Im}(\n\Delta^{-1}\n^*)}\subset H^{\perp}$ and $\Delta^{-1}\Delta=P_{H^{\perp}}$. Thus
\begin{align*}
	\Delta^{-1}\Delta-\n\Delta^{-1}\n^*=P_{\o{\text{Im}(\n\Delta^{-1}\n^*)}^{\perp}\cap H^{\perp}},
\end{align*}
which is the orthogonal projection from $A^{1}(M,u^*T_{\mb{C}}\mc{X}_z)$ to the space $\o{\text{Im}(\n\Delta^{-1}\n^*)}^{\perp}\cap H^{\perp}$.
Therefore, 
\begin{align*}
	\langle (\Delta^{-1}\Delta-\n\Delta^{-1}\n^*)s,s\rangle=\|P_{\o{\text{Im}(\n\Delta^{-1}\n^*)}^{\perp}\cap H^{\perp}}s\|^2\geq 0.
\end{align*}
\end{proof}

 \subsection{Harmonic maps}\label{subsec1.3}
For any smooth map $u: (M^n, g)\to (\mc{X}_z,\Phi)$
the differential
 $du$ is a section of the
bundle $T^*M\otimes u^*T_{\mb{C}}\mc{X}_z$. 
 Let $\{x^i\}$ denote a local coordinate system near a point $p$ in $M$
and $v$ the local complex coordinate on $\mc{X}_z$. Then
$du\in T^*M\otimes u^*T_{\mb{C}}\mc{X}_z$ is locally expressed as 
 \begin{align*}
 du=\frac{\p u^v}{\p x^i}dx^i\otimes \frac{\p}{\p v}+ \o{\frac{\p u^v}{\p x^i}}dx^i\otimes \frac{\p}{\p \b{v}}	
 \end{align*}
The energy density is given by
\begin{align*}
	|du|^2:=(du,du)=2g^{ij}u^v_i\o{u^v_j}\phi_{v\b{v}},
\end{align*}
where for convenience we  denote $u^v_i:=\frac{\p u^v}{\p x^i}$.
 The energy is defined by 
\begin{align}\label{energy}
E(u):=\frac{1}{2}\|du\|^2:=\frac{1}{2}\int_M |du|^2d\mu_g=\int_M (g^{ij}u^v_i\o{u^v_j}\phi_{v\b{v}})d\mu_g,	
\end{align}
where $d\mu_g=\sqrt{\det g}dx^1\wedge \cdots\wedge dx^n$. 
The harmonic equation for $u$ is 
\begin{aligns}\label{harmonic}
g^{ij}\left(\p_i u^v_j-\Gamma^{k}_{ij}u^v_k+(\p_v\log\phi_{v\b{v}})u^v_i u^v_j\right)\frac{\p}{\p v}=0;	
\end{aligns}
see e.g. \cite[Section 4.1]{Yamada1} or \cite[(1.2.10)]{Xin}. 

We recall that the  harmonicity of $u$ 
can be expressed in terms of harmonicity
of the form $du$, which we shall use. Note first
that dual operator $\nabla^\ast$
acts on  $f=f^v_idx^i\otimes \frac{\p}{\p v}\in A^1(M, u^*T\mc{X}_z)$ as
\begin{align}\label{nabla*}
\n^* f= \n^*(f^v_i\frac{\p}{\p v}\otimes dx^i)=-(g^{ij}\n_j f^v_i)\frac{\p}{\p v}.	
\end{align}
In fact for any $e\in A^0(M, u^*T\mc{X}_z)$
\begin{align*}\langle \n^* f, e\rangle=-\int_M g^{ij}\n_j f^v_i \o{e^v}\phi_{v\b{v}}d\mu_g=\int_M g^{ij}f^v_i \n_j\o{e^v}\phi_{v\b{v}}d\mu_g=\langle f,\n e\rangle.\end{align*}
 Thus the harmonic equation (\ref{harmonic}) is equivalent to 
 \begin{align}\label{1.5}
 \n^*\left(u^v_j  dx^j\otimes \frac{\p}{\p v}\right)=-(g^{ij}\n_i u^v_j)\frac{\p}{\p v}=0.	
 \end{align}
On the other hand by a direct calculation 
\begin{aligns}\label{1.6}
\n	\left(u^v_j  dx^j\otimes \frac{\p}{\p v}\right)&=(\n_i u^v_j)dx^i\wedge dx^j\otimes \frac{\p}{\p v}\\
&=(\p_iu^v_j+\phi_v u^v_i u^v_j-\Gamma^k_{ij}u^v_k)dx^i\wedge dx^j\otimes\frac{\p}{\p v}=0. 
\end{aligns}
Combining (\ref{1.5}) with (\ref{1.6}), we obtain
\begin{prop}[{\cite[Proposition 1.3.3]{Xin}}]\label{prop1}
$u$ is a harmonic map if and only if $du$ is harmonic, i.e. $\Delta du=0$.	
\end{prop}

We shall also need the following two theorems; see e.g.
 \cite[Section 4.1]{Yamada1} and references therein. 
\begin{thm}[\cite{ES, Hartman, Albers}]\label{thm1.1} Let $(M^n,g)$ be a closed Riemannian manifold, and $(\Sigma^2,\Phi)$ a surface of non-positive sectional curvature. Suppose there is a continuous map $u_0:(M^n,g)\to (\Sigma^2,\Phi)$. Then there exists a smooth harmonic map homotopic to $u_0$. When the sectional curvature of $\Phi$ is strictly negative and the image of the map is not a point or a closed geodesic, then the harmonic map is unique.
\end{thm}
\begin{thm}[ \cite{EL, Koiso}]\label{thm1.2}
Let $(M^n,g)$ be a closed Riemannian manifold, and $(\Sigma^2,\Phi)$ a closed surface with a hyperbolic metric $\Phi$. For a smooth deformation $\Phi_t$ of the hyperbolic metric $\Phi:=\Phi_0$ in the space of smooth metrics on $\Sigma$, the resulting harmonic maps $u_t:(M^n,g)\to (\Sigma^2,\Phi_t)$ are smoothly depending in $t$.
\end{thm}

\section{Variations of energy on Teichm\"uller space}\label{sec2}

In this section we will compute the first and the second variations of 
the energy $E(u(z))$ for harmonic maps $u:M^n\to \mc{X}_{z}$.
 Fixed a smooth map $u_0:M^n\to \mc{X}_{z_0}$, $z_0\in\mc{T}$. From Theorem \ref{thm1.1}, \ref{thm1.2} and \cite[Section 1.1]{Yamada}, the following 
 function
\begin{align}
E(z):=E(u(z))	
\end{align}
is well-defined and smooth 
on Teichm\"uller space $\mc{T}$, where $u(z)$ is a harmonic map from $(M^n, g)\to (\mc{X}_z,\Phi)$ and homotopic to $u_0$. 
In order to find the variations $\frac{\p }{\p z^{\alpha}}E(z)$ and $\frac{\p^2}{\p z^{\alpha}\p\b{z}^{\beta}}E(z)$ it is enough to compute 
$$(\p E(z))(\xi)=\frac{\p E(z)}{\p  z^{\alpha}}\xi^{\alpha},\quad \p\b{\p}E(z)(\xi,\o{\xi})=\frac{\p^2E(z)}{\p z^{\alpha}\p\b{z}^{\beta}}\xi^{\alpha}\b{\xi}^{\beta}$$
along a single direction $\xi=\xi^{\alpha}\frac{\p}{\p z^{\alpha}}\in T\mc{T}$. 
So with some abuse of notation
we assume that the base manifold $\mc{T}$ is one dimensional
with  $z$ as local holomorphic coordinate, and
the indices $\alpha$ and  $\b \beta$ above will be replaced by
$z$ and $\b{z}$.

\subsection{The first variation}\label{subsec2.1}

Recall  the notation (\ref{Aalpha}) and define
\begin{align}\label{defA}
  A:=A_z=
  A_{z\b{v}\b{v}}\o{u^v_i}\phi^{v\b{v}}dx^i\otimes \frac{\p}{\p v}
  =  A_{z\b{v}
  }^v\o{u^v_i}
  dx^i\otimes \frac{\p}{\p v}
  \in A^1(M, u^*T\mc{X}_z). 
\end{align}
It will play an important role in the variation
formulas below. Note that $A$ is the pull-back of 
the Kodaira-Spencer tensor
(\ref{A-pre}).

\begin{thm}\label{thm2.1}
The first variation of the energy function $E(z)$ is given by 	
\begin{align}
\frac{\p}{\p z}E(z)=\langle A, du\rangle.
\end{align}
\end{thm}
\begin{proof}
  We perform
  the differentiation $  \frac{\p}{\p z}$
  on the definition of the energy (\ref{energy}),
\begin{align}\label{1.12}
  \frac{\p}{\p z}
  E(z)=\int_M \left(g^{ij}(\p_z u^v_i)\o{u^v_j}\phi_{v\b{v}}+g^{ij}u^v_i\o{\p_{\b{z}}u^v_j}\phi_{v\b{v}}+g^{ij}u^v_i \o{u^v_j}\p_z\phi_{v\b{v}}\right)d\mu_g.
\end{align}
The family of harmonic maps $u(z)$ will be treated as a map  
\begin{align}\label{family harmonic maps}
U: \mc{T}\times M\to \mc{X},\quad U(z,x)=(z, v=u(z,x)).	
\end{align}
The pull-back of $\phi$ is  $\phi=\phi(z, u(z,x))$, so  that
\begin{align}\label{1.13}
	\p_z\phi_{v\b{v}}=\phi_{v\b{v}z}+\phi_{v\b{v}v}u^v_z+\phi_{v\b{v}\b{v}}\o{u^v_{\b{z}}}.
\end{align}
Substituting (\ref{1.13}) into (\ref{1.12}), and
 using $\nabla_i u_z^v= \p_iu_{z}^v+\phi_v u_z^v u_i^v$,
we obtain
\begin{align*}
	\frac{\p}{\p z}E(z)&=\int_M \left(g^{ij}(\n_i u^v_z)\o{u^v_j}\phi_{v\b{v}}+g^{ij}u^v_i\o{\n_{j}u^v_{\b{z}}}\phi_{v\b{v}}+g^{ij}u^v_i \o{u^v_j}\phi_{v\b{v}z}\right)d\mu_g\\
	&=\int_M \left(-g^{ij} u^v_z\n_i\o{u^v_j}\phi_{v\b{v}}-g^{ij}\n_{j}u^v_i\o{u^v_{\b{z}}}\phi_{v\b{v}}+g^{ij}u^v_i \o{u^v_j}\phi_{v\b{v}z}\right)d\mu_g\\
                           &=\int_M g^{ij}u^v_i
                             \o{u^v_j}\phi_{v\b{v}z}
                             d\mu_g,
\end{align*}
where the last equality follows from the harmonic equation (\ref{1.5}).
The factor $u^v_i\phi_{v\b{v}z}$ can be expressed
in term of $
A_{z\b{v}\b{v}}\o{u^v_i}$,
by Lemma \ref{lemma1} (ii), as follows, 
\begin{align*}
  \n_i \phi_{z\b{v}}=\phi_{v\b{v}z}u^v_i+\phi_{z\b{v};\b{v}}\o{u^v_i}=\phi_{v\b{v}z}u^v_i-
  A_{z\b{v}\b{v}}\o{u^v_i}.	
\end{align*}
Thus, again by the harmonic equation (\ref{1.5}), we obtain
\begin{aligns}\label{F1}
	\frac{\p}{\p z}E(z)&=\int_M g^{ij}u^v_i \o{u^v_j}\phi_{v\b{v}z}d\mu_g\\
	&=\int_M g^{ij} \o{u^v_j}\n_i\phi_{z\b{v}}d\mu_g+\int_M A_{z\b{v}\b{v}}\o{u^v_i}\o{u^v_j}g^{ij}d\mu_g\\
	&=-\int_M g^{ij} \n_i\o{u^v_j}\phi_{z\b{v}}d\mu_g+\int_M A_{z\b{v}\b{v}}\o{u^v_i}\o{u^v_j}g^{ij}d\mu_g\\
	&=\int_M A_{z\b{v}\b{v}}\o{u^v_i}\o{u^v_j}g^{ij}d\mu_g.
\end{aligns}
On the other hand
\begin{aligns}\label{F2}
	\langle A, du\rangle &=\left\langle A_{z\b{v}\b{v}}\o{u^v_i}\phi^{v\b{v}}dx^i\otimes \frac{\p}{\p v}, u^v_i dx^i\otimes \frac{\p}{\p v}+\o{u^v_i}dx^i\otimes \frac{\p}{\p\b{v}}\right\rangle\\
	&=\int_M A_{z\b{v}\b{v}}\o{u^v_i}\o{u^v_j}g^{ij}d\mu_g,
\end{aligns}
which is $	\frac{\p}{\p z}E(z)$, completing the proof.

\end{proof}

\subsection{The second variation}\label{subsec2.2}
 We shall use the method in \cite{AS} where the case
$M$ being the unit circle, namely $u$ being a closed geodesic,
is considered.
\begin{lemma}\label{lemma3}
 We have
\begin{multline*}\frac{\p}{\p\b{z}}\left(A_{z\b{v}\b{v}}(z,u(z,x))\o{u^v_i}\o{u^v_j}g^{ij}\right)=(-c(\phi)_{z\b{z};\b{v}\b{v}}-2A_{z\b{v}\b{v}}\o{a^v_{z;v}}-A_{z\b{v}\b{v};\b{v}}\o{a^v_z})\o{u^v_i}\o{u^v_j}g^{ij}\\
	+\p_{\b{v}}A_{z\b{v}\b{v}}\o{u^v_z}\o{u^v_i}\o{u^v_j}g^{ij}+2A_{z\b{v}\b{v}}\o{\p_z u^v_i}\o{u^v_j}g^{ij}.\end{multline*}
\end{lemma}
\begin{proof}
From (\ref{family harmonic maps}), we have  
\begin{align*}\frac{\p}{\p\b{z}}A_{z\b{v}\b{v}}(z,u(z,x))=(\p_{\b{z}}A_{z\b{v}\b{v}})(z,u)+\p_{\b{v}}A_{z\b{v}\b{v}}\o{u^v_z}+\p_vA_{z\b{v}\b{v}}u^v_{\b{z}}.\end{align*}
This combined with Lemma \ref{lemma1} (iii)-(iv) gives
\begin{align*}
	\frac{\p}{\p\b{z}}\left(A_{z\b{v}\b{v}}(z,u(z,x))\o{u^v_i}\o{u^v_j}g^{ij}\right) =&(-c(\phi)_{z\b{z};\b{v}\b{v}}-2A_{z\b{v}\b{v}}\o{a^v_{z;v}}-A_{z\b{v}\b{v};\b{v}}\o{a^v_z})\o{u^v_i}\o{u^v_j}g^{ij}\\
	&+\p_{\b{v}}A_{z\b{v}\b{v}}\o{u^v_z}\o{u^v_i}\o{u^v_j}g^{ij}+2A_{z\b{v}\b{v}}\o{\p_z u^v_i}\o{u^v_j}g^{ij}.
\end{align*}
\end{proof}
We recall the definition of divergence of $\alpha$
for any one form $\alpha=\alpha_i dx^i\in A^1(M)$, 
\begin{align}\label{div}
\text{div}(\alpha)=g^{ij}\n_i \alpha_j,
\end{align}
and Stokes' theorem that
\begin{align}
	\int_M \text{div}(\alpha)d\mu_g=0.
\end{align}
Let 
\begin{align}\label{W-0}
W=P_{\mathcal V}
U_*(\frac{\partial}
{\partial z}
) 
=u_z^v 
\frac{\partial}
{\partial v}
+\phi_{\bar v z} \phi^{v\bar v}
\frac{\partial}{\partial v}=(u_z^v- a_z^v)\frac{\partial}{\partial v} 
\end{align}
be the vertical  projection  of push-forward 
$U_*(\frac{\partial}
{\partial z}
)$, 
and
\begin{aligns}\label{alpha}
\alpha =(A, W)
	=A_{z\b{v}\b{v}}\o{u^v_i}(\o{u^v_z}-\o{a^v_z})dx^i.	
\end{aligns} 
\begin{lemma}\label{lemma4}
If $\alpha$ is given by (\ref{alpha}), then 	
\begin{multline*}
\text{div}(\alpha) =g^{ij}A_{z\b{v}\b{v};\b{v}}\o{u^v_i}\o{u^v_j}(\o{u^v_z}-\o{a^v_z})
\\+g^{ij}A_{z\b{v}\b{v}}\o{u^v_j}(\o{\p_i u^v_z+(\p_v\log\phi_{v\b{v}})u^v_z u^v_i-A^{v}_{z\b{v}}\o{u^v_i}-a^{v}_{z;v}u^v_i}).
\end{multline*}
\end{lemma}
\begin{proof}
By the definition of $\text{div}(\alpha)$	in (\ref{div}), we have 
\begin{align*}
\text{div}(\alpha) &=g^{ij}\n_i(A_{z\b{v}\b{v}}\o{u^v_j}(\o{u^v_z}-\o{a^v_z}))\\
&=g^{ij}\left((\n_i A_{z\b{v}\b{v}})\o{u^v_j}(\o{u^v_z}-\o{a^v_z})+A_{z\b{v}\b{v}}\o{u^v_j}\n_i(\o{u^v_z}-\o{a^v_z})+A_{z\b{v}\b{v}}(\o{u^v_z}-\o{a^v_z})\n_i\o{u^v_j}\right)\\
&=g^{ij}\left((\n_i A_{z\b{v}\b{v}})\o{u^v_j}(\o{u^v_z}-\o{a^v_z})+A_{z\b{v}\b{v}}\o{u^v_j}\n_i(\o{u^v_z}-\o{a^v_z})\right),
\end{align*}
where the last equality follows from harmonic equation (\ref{1.5}). Using Lemma \ref{lemma1} (ii), we find
\begin{align*}\n_iA_{z\b{v}\b{v}}=A_{z\b{v}\b{v};v}u^v_i+A_{z\b{v}\b{v};\b{v}}\o{u^v_i}=A_{z\b{v}\b{v};\b{v}}\o{u^v_i}\end{align*}
and 
\begin{align*}
\n_i(u^v_z-a^v_z)=\p_i u^v_z+(\p_v \log\phi_{v\b{v}})u^v_z u^v_i-A^{v}_{z\b{v}}\o{u^v_i}-a^{v}_{z;v}u^v_i,	
\end{align*}
so
\begin{multline*}
\text{div}(\alpha) =g^{ij}A_{z\b{v}\b{v};\b{v}}\o{u^v_i}\o{u^v_j}(\o{u^v_z}-\o{a^v_z})
\\+g^{ij}A_{z\b{v}\b{v}}\o{u^v_j}(\o{\p_i u^v_z+(\p_v\log\phi_{v\b{v}})u^v_z u^v_i-A^{v}_{z\b{v}}\o{u^v_i}-a^{v}_{z;v}u^v_i}).
\end{multline*}
\end{proof}

\begin{lemma}
  The second variation
  $\frac{\p^2}{\p z\p\b{z}}E(z)$   is
\begin{aligns}\label{2.3}
 &\quad \frac{\p^2}{\p z\p\b{z}}E(z)\\
  &=\frac{\p}{\p\b{z}}\int_M A_{z\b{v}\b{v}}\o{u^v_i}\o{u^v_j}g^{ij}d\mu_g\\
&=\int_M \left(\frac{\p}{\p\b{z}}(A_{z\b{v}\b{v}}\o{u^v_i}\o{u^v_j}g^{ij})	-2\text{div}(\alpha)+\text{div}(c(\phi)_{z\b{z};\b{v}}\o{u^v_j}dx^j)\right)d\mu_g\\
&=\int_M \left(c(\phi)_{z\b{z}}g^{ij}\phi_{v\b{v}}u^v_i\o{u^v_j}+g^{ij}A_{z\b{v}\b{v}}A^{\b{v}}_{\b{z}v}u^v_i\o{u^v_j}-g^{ij}\n_i A_{z\b{v}\b{v}}\o{u^v_j}(\o{u^v_z}-\o{a^v_z})\right)d\mu_g. 
\end{aligns}
\end{lemma}
  \begin{proof}
 Similar computations as above give
\begin{aligns}\label{2.1}
\text{div}(c(\phi)_{z\b{z};\b{v}}\o{u^v_j}dx^j) &=g^{ij}\n_i(c(\phi)_{z\b{z};\b{v}}\o{u^v_j})=g^{ij}(\n_ic(\phi)_{z\b{z};\b{v}})\o{u^v_j}	\\
&=g^{ij}c(\phi)_{z\b{z};v\b{v}}u^v_i\o{u^v_j}+g^{ij}c(\phi)_{z\b{z};\b{v}\b{v}}\o{u^v_i}\o{u^v_j}.
\end{aligns}
Adding up the formulas
in Lemmas \ref{lemma3}, \ref{lemma4} and (\ref{2.1}) results in
\begin{aligns}\label{2.2}
&\quad \frac{\p}{\p\b{z}}(A_{z\b{v}\b{v}}\o{u^v_i}\o{u^v_j}g^{ij})	-2\text{div}(\alpha)+\text{div}(c(\phi)_{z\b{z};\b{v}}\o{u^v_j}dx^j)\\
&=(-c(\phi)_{z\b{z};\b{v}\b{v}}-2A_{z\b{v}\b{v}}\o{a^v_{z;v}}-A_{z\b{v}\b{v};\b{v}}\o{a^v_z})\o{u^v_i}\o{u^v_j}g^{ij}\\
	&\quad+\p_{\b{v}}A_{z\b{v}\b{v}}\o{u^v_z}\o{u^v_i}\o{u^v_j}g^{ij}+2A_{z\b{v}\b{v}}\o{\p_z u^v_i}\o{u^v_j}g^{ij}-2g^{ij}A_{z\b{v}\b{v};\b{v}}\o{u^v_i}\o{u^v_j}(\o{u^v_z}-\o{a^v_z})\\
&\quad -2g^{ij}A_{z\b{v}\b{v}}\o{u^v_j}(\o{\p_i u^v_z+(\p_v\log\phi_{v\b{v}})u^v_z u^v_i-A^{v}_{z\b{v}}\o{u^v_i}-a^{v}_{z;v}u^v_i})\\
	&\quad+g^{ij}c(\phi)_{z\b{z};v\b{v}}u^v_i\o{u^v_j}+g^{ij}c(\phi)_{z\b{z};\b{v}\b{v}}\o{u^v_i}\o{u^v_j}\\
	&=g^{ij}c(\phi)_{z\b{z};v\b{v}}u^v_i\o{u^v_j}+2g^{ij}A_{z\b{v}\b{v}}A^{\b{v}}_{\b{z}v}u^v_i\o{u^v_j}-g^{ij}A_{z\b{v}\b{v};\b{v}}\o{u^v_i}\o{u^v_j}(\o{u^v_z}-\o{a^v_z})\\
	&=c(\phi)_{z\b{z}}g^{ij}\phi_{v\b{v}}u^v_i\o{u^v_j}+g^{ij}A_{z\b{v}\b{v}}A^{\b{v}}_{\b{z}v}u^v_i\o{u^v_j}-g^{ij}\n_i A_{z\b{v}\b{v}}\o{u^v_j}(\o{u^v_z}-\o{a^v_z}),
\end{aligns}
where in the second equality we used 
$A_{z\b{v}\b{v};\b{v}}=\p_{\b{v}}A_{z\b{v}\b{v}}-2A_{z\b{v}\b{v}}\p_{\b{v}}\log\phi_{v\b{v}},$
the last equality follows from Lemma \ref{lemma1} (iii) (v).
Our lemma now follows from Theorem \ref{thm2.1},  (\ref{div}) and (\ref{2.2}).
  \end{proof}

Now we set 
\begin{align}\label{N1}
	& V=P_{\mc{V}}U_*(\frac{\p}{\p\b{z}})=P_{\mc{V}}\left(\frac{\p}{\p\b{z}}+\o{u^v_z}\frac{\p}{\p\b{v}}+u^v_{\b{z}}\frac{\p}{\p v}\right)=u^v_{\b{z}}\frac{\p}{\p v};\\\label{N2}
	&\mc{L}=\Delta+\frac{1}{2}|du|^2=\Delta+g^{ij}\phi_{v\b{v}}u^v_i\o{u^v_j};\\\label{N3}
	&\mc{G}=g^{ij}\phi_{v\b{v}}u^v_i u^v_j \frac{\p}{\p v}\otimes d\b{v}\in \text{Hom}(u^*\o{T\mc{X}_z},u^*T\mc{X}_z).
\end{align}
By conjugation, $\o{\mc{G}}=g^{ij}\phi_{v\b{v}}\o{u^v_i} \o{u^v_j} \frac{\p}{\p \b{v}}\otimes dv\in \text{Hom}(u^*T\mc{X}_z,u^*\o{T\mc{X}_z)}$.
\begin{lemma}\label{lemma5}
	We have 
	\begin{itemize}
	\item[(i)] $\n A=0$;
	\item[(ii)] $\mc{L}(W)=\mc{G}(\o{V})-\n^*A$;
	\item[(iii)] $\mc{L}(\o{V})=\o{\mc{G}}(W)$.
	\end{itemize}
\end{lemma}
\begin{proof}
(i) By the definition of $\n$ in (\ref{nabla}), Lemma \ref{lemma1} (ii) and (\ref{defA}), $\n A$ is
\begin{align*}
\n A
&=\n\left(A_{z\b{v}\b{v}}\o{u^v_l}\phi^{v\b{v}}dx^l\otimes \frac{\p}{\p v}\right)	\\
&=\n_i\left(A_{z\b{v}\b{v}}\o{u^v_l}\phi^{v\b{v}}\right)	dx^i\wedge dx^l\otimes \frac{\p}{\p v}\\
&=\left(A_{z\b{v}\b{v};\b{v}}\o{u^v_l}\o{u^v_i}\phi^{v\b{v}}+A_{z\b{v}\b{v}}(\p_i\p_l\o{u^v}-\o{\Gamma^k_{il}}\o{u^v_k}+\phi_{\b{v}}\o{u^v_l}\o{u^v_i})\phi^{v\b{v}}\right)dx^i\wedge dx^l\otimes \frac{\p}{\p v}\\
&=0.
\end{align*}
Note that the last equality follows as follows:  If we set $$\alpha_{il}=A_{z\b{v}\b{v};\b{v}}\o{u^v_l}\o{u^v_i}\phi^{v\b{v}}+A_{z\b{v}\b{v}}(\p_i\p_l\o{u^v}-\o{\Gamma^k_{il}}\o{u^v_k}+\phi_{\b{v}}\o{u^v_l}\o{u^v_i})\phi^{v\b{v}},$$ then $\alpha_{il}=\alpha_{li}$, hence $\alpha_{il} dx^i\wedge dx^l=\alpha_{il}dx^l\wedge dx^i$, which implies that $\alpha_{il} dx^i\wedge dx^l=0$.

(ii) When acting on $W\in A^0(M, u^*T\mc{X}_z)$ given
in (\ref{W-0}), $-\Delta=g^{ij}\n_i\n_j$, and
\begin{align}
-\Delta W=g^{ij}(u^v_z-a^v_z)_{;ji}	\frac{\p}{\p v}. 
\end{align}
The coefficient above, by Lemma \ref{lemma1} (i)-(iii)  and
$-a_{z;v}^v=\phi_z$, is
\begin{aligns}\label{2.4}
  &\quad
  g^{ij}(u^v_z-a^v_z)_{;ji} \\
  &=g^{ij}(\p_j u^v_z+\phi_v u^v_z u^v_j-A^{v}_{z\b{v}}\o{u^v_j}-a^v_{z;v}u^v_j)_{;i}\\
	 &=g^{ij}\left[\p_i\p_j u^v_z-\p_k u^v_z\Gamma^k_{ij}+\p_ju^v_z\phi_v u^v_i\right.\\
	&\quad\left. +(\phi_{vv}u^v_i+\phi_{v\b{v}}\o{u^v_i}-\phi_v\phi_v u^v_i)u^v_zu^v_j\right.\\
	 &\quad \left.+\phi_v(\p_i u^v_z+u^v_z\phi_v u^v_i)u^v_j-A^{v}_{z\b{v};\b{v}}\o{u^v_i}\o{u^v_j}+\phi_{zv}u^v_iu^v_j+\phi_{z\b{v}}\o{u^v_i}u^v_j\right].
\end{aligns}
The first three terms in RHS of (\ref{2.4}), by (\ref{1.5}), are
\begin{aligns}\label{2.5}
	&\quad g^{ij}\left(\p_i\p_j u^v_z-\p_k u^v_z\Gamma^k_{ij}+\p_ju^v_z\phi_v u^v_i\right)\\
	&=\p_z\left[g^{ij}(\p_i\p_j u^v-\p_k u^v\Gamma^k_{ij}+\p_ju^v\phi_v u^v_i)\right]\\
	&\quad-g^{ij}u^v_j\phi_v \p_z u^v_i-g^{ij}u^v_j\p_z\phi_v u^v_i\\
	&=\p_z(g^{ij}\n_j u^v_i)-g^{ij}u^v_j\phi_v \p_z u^v_i-g^{ij}u^v_ju^v_i\p_z\phi_v(z,u(z,x)) \\
	&=g^{ij}\left[-(\phi_{zv}+\phi_{v\b{v}}\o{u^v_{\b{z}}}+\phi_{vv}u^v_z)u^v_iu^v_j-u^v_j\phi_v\p_z u^v_i\right].
\end{aligns}
Substituting (\ref{2.5}) into (\ref{2.4}),  we obtain
\begin{aligns}\label{2.6}
  &\quad g^{ij}(u^v_z-a^v_z)_{;ji}\\
  &=g^{ij}\left[-(\phi_{zv}+\phi_{v\b{v}}\o{u^v_{\b{z}}}+\phi_{vv}u^v_z)u^v_iu^v_j-u^v_j\phi_v\p_z u^v_i\right.\\
	&\quad\left. +(\phi_{vv}u^v_i+\phi_{v\b{v}}\o{u^v_i}-\phi_v\phi_v u^v_i)u^v_zu^v_j\right.\\
	 &\quad \left.+\phi_v(\p_i u^v_z+u^v_z\phi_v u^v_i)u^v_j-A^{v}_{z\b{v};\b{v}}\o{u^v_i}\o{u^v_j} +\phi_{zv}u^v_iu^v_j+\phi_{z\b{v}}\o{u^v_i}u^v_j\right]\\
	&=(g^{ij}\phi_{v\b{v}}\o{u^v_i}u^v_j) u^v_z-g^{ij}A^{v}_{z\b{v};\b{v}}\o{u^v_i}\o{u^v_j}+g^{ij}\phi_{z\b{v}}\o{u^v_i}u^v_j-g^{ij}\phi_{v\b{v}}\o{u^v_{\b{z}}}u^v_iu^v_j.\\
	&=(g^{ij}\phi_{v\b{v}}\o{u^v_i}u^v_j)(u^v_z-a^v_z)-g^{ij}\phi_{v\b{v}}\o{u^v_{\b{z}}}u^v_iu^v_j-g^{ij}A_{z\b{v}\b{v};\b{v}}\o{u^v_i}\o{u^v_j}\phi^{v\b{v}}.
\end{aligns}
By (\ref{nabla*}), (\ref{defA}), (\ref{W-0}), (\ref{N1})-(\ref{N3}) and (\ref{2.6}), we get
\begin{align*}
	\mc{L}(W) &=\mc{L}\left((u^v_z-a^v_z)\frac{\p}{\p v}\right)\\
	&=\left(g^{ij}\phi_{v\b{v}}\o{u^v_{\b{z}}}u^v_iu^v_j+g^{ij}\n_i A_{z\b{v}\b{v}}\o{u^v_j}\phi^{v\b{v}}\right)\frac{\p}{\p v}\\
	&=\mc{G}(\o{V})-\n^*A.
\end{align*}
(iii) Similarly, by a direct calculation, 
$\nabla_i \phi_v=\phi_{vv}u_i^v+\phi_{v\bar v}\overline{u_i^v}-\phi_v\phi_v u_i^v $,
and
\begin{align*}
 g^{ij}(u^v_{\b{z}})_{;ji}&=g^{ij}\left(\p_j u^v_{\b{z}}+u^v_{\bar z}\phi_v u^v_j\right)_{;i}\\
 &=g^{ij}\left[\p_i\p_ju^v_{\b{z}}-\p_k u^v_{\b{z}}\Gamma^k_{ij}+\p_j u^v_{\b{z}}\phi_v u^v_i+(\p_i u^v_{\b{z}}+u^v_{\b{z}}\phi_v u^v_i)\phi_v u^v_j\right.\\
 &\quad\left.+ u^v_{\b{z}}u^v_j\left(\phi_{vv}u^v_i+\phi_{v\b{v}}\o{u^v_i}-\phi_v\phi_v u^v_i\right)\right]	\\
 &=g^{ij}\left[-(\phi_{v\b{z}}+\phi_{vv}u^v_{\b{z}}+\phi_{v\b{v}}\o{u^v_{z}})u^v_iu^v_j-\p_{\b{z}}u^v_i u^v_j\phi_v\right.\\
 &\quad\left.+(\p_i u^v_{\b{z}}+u^v_{\b{z}}\phi_v u^v_i)\phi_v u^v_j\right.\\
 &\quad\left.+ u^v_{\b{z}}u^v_j\left(\phi_{vv}u^v_i+\phi_{v\b{v}}\o{u^v_i}-\phi_v\phi_v u^v_i\right)\right]\\
 &=(g^{ij}\phi_{v\b{v}}\o{u^v_i}u^v_j) u^v_{\b{z}}-g^{ij}\phi_{v\b{z}}u^v_iu^v_j-g^{ij}\phi_{v\b{v}}\o{u^v_{z}}u^v_i u^v_j\\
 &=(g^{ij}\phi_{v\b{v}}\o{u^v_i}u^v_j) u^v_{\b{z}}-g^{ij}\phi_{v\b{v}}u^v_i u^v_j(\o{u^v_z}-\o{a^v_z}),
 \end{align*}
 where the third equality follows from (\ref{2.5}) by replacing $z$ by $\b{z}$. 
By conjugation, we conclude that 
\begin{align}
\mc{L}(\o{V})=\mc{L}\left( \o{u^v_{\b{z}}}\frac{\p}{\p\b{v}}\right)=\left(g^{ij}\phi_{v\b{v}}\o{u^v_i} \o{u^v_j}(u^v_z-a^v_z)\right)\frac{\p}{\p\b{v}}=\o{\mc{G}}(W).
\end{align}
\end{proof}
\begin{rem}\label{rem2.5}
The formulas (ii) and (iii) above can also be proved easily
  by choosing a normal coordinates $x^j$
  near $x_0$ and
  holomorphic normal coordinate $v$ at $v_0=u(z_0, x_0)$.
We sketch the proof of (ii) here. 
The Christoffel symbol
on the Riemann surface $\mathcal X_z$ is
$\Gamma_{vv}^v=
\partial_v \log \phi_{v\bar v}=\phi_{v}$ and $\phi_v=\partial_v\phi_v=0$  at $v_0$, and $\Gamma^i_{jk}=0$ at $x_0\in M$.
Denote  $\nabla$ also the connection
on $u^*T \mathcal X_z\otimes T^*M$. We have
$$
\nabla W
=d(u_z^v-a_z^v)\otimes \frac{\partial}{\partial v}
+(u_z^v-a_z^v)\otimes \phi_v du^v\frac{\partial}{\partial v}
$$
and 
\begin{equation*}
  \begin{split}
\nabla
\nabla W
&=\nabla (d(u_z^v-a_z^v))
\otimes \frac{\partial}{\partial v}
+ d(u_z^v-a_z^v)\otimes \phi_v du^v\frac{\partial}{\partial v}
+d((u_z^v-a_z^v)\phi_v)\otimes
du^v \otimes \frac{\partial}{\partial v}
\\
&+ (u_z^v-a_z^v)\phi_v\otimes
\nabla(du^v) \otimes \frac{\partial}{\partial v}
+ (u_z^v-a_z^v)\phi_v du^v
\otimes\phi_v du^v \frac{\partial}{\partial v}.    
  \end{split}
\end{equation*}
Evaluating it at $v_0=u(z_0, x_0)$ we get
\begin{equation*}
  \begin{split}
\nabla
\nabla W
&=\nabla (d(u_z^v-a_z^v))
\otimes \frac{\partial}{\partial v}
+d((u_z^v-a_z^v) \phi_v) \otimes
du^v \otimes \frac{\partial}{\partial v}.
  \end{split}
\end{equation*}
At $v_0$ the differential $d((u_z^v-a_z^v) \phi_v)
=d(u_z^v-a_z^v) \phi_v +(u_z^v-a_z^v) d\phi_v 
=(u_z^v-a_z^v) \partial_v\phi_v du^v +
(u_z^v-a_z^v) \partial_{\bar v} \phi_v d\bar u^v 
$ and 
$\Delta W
=-\text{Tr}_g\nabla \nabla W$ is
\begin{equation}
  \label{eq:eval-W-0}
\left(\Delta_g (u_z^v) -
\Delta_g (a_z^v)  - (u_z^v -a_z^v)\phi_{v\bar v} 
\text{Tr}_g (d \bar{u^v} \otimes du^v)\right)\frac{\partial}{\partial v}.  
\end{equation}
Here $\text{Tr}_g$ is the trace function on $TM\otimes TM$
with respect to the metric $g=(g^{ij})$.
Differentiating the harmonic equation $\text{Tr}_g \nabla du^v=0$ in $z$ and
evaluated at $v_0=u(z_0, x_0)$ we find
$$
\Delta_g (u_z^v) =g^{ij}(
 \partial_{\bar v}
\Gamma_{v v}^v
 u^{\bar v}_z
+\partial_z \Gamma_{v v}^v 
)u_i^v u_j^v
$$
since 
$\Gamma_{v v}^v=
\partial_v \log \phi_{v\bar v}=\phi_v$.
 The first term above gives
$$
g^{ij}\phi_{v\b{v}}\o{u^v_{\b{z}}}u^v_iu^v_j\frac{\p}{\p v} =\mathcal{G}(\bar V), 
$$
and  the second term
is
\begin{align}\label{2.17}
g^{ij}\partial_z \Gamma_{v v}^v u^v_i u^v_j
=g^{ij}\phi_{vz} u^v_i u^v_j.
\end{align}
The term
$\Delta_g (a_z^v)$  is
$$
\Delta_g (a_z^v)
=-\text{Tr}_g \nabla_g d a_z^v
=-\text{Tr}_g \nabla_g 
\left(
\partial_v(a_z^v) u_i^v dx^i
+ \partial_{\bar v}(a_z^v) \overline{u_i^v} dx^i
\right),
$$
the second term
is 
$$
-\text{Tr}_g \nabla_g  (\partial_{\bar v}(a_z^v) \overline{u_i^v} dx^i)\frac{\p}{\p v}
=-\nabla_i (A^v_{z\b{v}}\b{u}^v_i) \frac{\p}{\p v}
=\nabla^* A
$$
and the first term, using the harmonicity of $u$
with normal coordinates $(x^j, v)$, is 
$$-\text{Tr}_g \nabla_g 
\partial_v(a_z^v) u_i^v dx^i=g^{ij}\phi_{zv}u^v_i u^v_j$$
at $v_0=u(z_0, x_0)$,
which is canceled by (\ref{2.17}); we omit the details here.
Finally the third term in
$(\ref{eq:eval-W-0})$ is
$$
-\frac 12 |du|^2 W.
$$
Thus
$$
\Delta W
=\mathcal{G}(\bar V) -\nabla^* A 
-\frac 12 |du|^2 W,
$$
i.e.,
$$
\mathcal L (W)=\Delta W
+\frac 12 |du|^2 W
=\mathcal{G}(\bar V) -\nabla^* A.
$$
This completes the proof of (ii).
\end{rem}
\begin{lemma}\label{lemma6}
  The operators $\mc{L}-\mc{G}\mc{L}^{-1}\o{\mc{G}}$ and $\frac{1}{2}|du|^2-\mc{G}\mc{L}^{-1}\o{\mc{G}}$	are non-negative and symmetric   when acting on 
  $ A^{0}(M, u^*T\mc{X}_z)$, and 
\begin{align}\label{2.10}
\text{Ker}\left(\mc{L}-\mc{G}\mc{L}^{-1}\o{\mc{G}}\right)\subset H=\text{Ker}\Delta.
\end{align}
\end{lemma}
\begin{proof}
  Note first that  $\langle \mc{L} e, e \rangle >0$ for any $e\neq 0\in A^{0}(M, u^*T\mc{X}_z)$. Hence $\mc{L}^{-1}$ is well-defined.
  Note that $\mathcal L \ge \frac 12
  \frac{1}{2}|du|^2 $ as symmetric operators
on   $A^{0}(M, u^*T\mc{X}_z)$, so that for  $e \in A^{0}(M, u^*T\mc{X}_z)$, 
\begin{aligns}\label{2.8}
  &\langle (\mc{L}-{\mc{G}}\mc{L}^{-1}\o{\mc{G}}) e, e \rangle
  \\
  &\geq \langle (\frac{1}{2}|du|^2-{\mc{G}}\mc{L}^{-1}\o{\mc{G}}) e, e \rangle\\
 &= \langle (g^{ij}\phi_{v\b{v}}u^v_i\o{u^v_j}-{\mc{G}}\mc{L}^{-1}\o{\mc{G}}) e, e \rangle\\
&=\langle (g^{ij}\phi_{v\b{v}}u^v_i\o{u^v_j})e,e\rangle-\langle(g^{ij}\phi_{v\b{v}}u^v_i\o{u^v_j}+\Delta)^{-1}\o{\mc{G}}e,\o{\mc{G}}e\rangle\\
&\geq \int_M 	(g^{ij}\phi_{v\b{v}}u^v_i\o{u^v_j}-(g^{ij}\phi_{v\b{v}}u^v_i\o{u^v_j})^{-1}(g^{ij}\phi_{v\b{v}}\o{u^v_i}\o{u^v_j}g^{kl}\phi_{v\b{v}}u^v_k u^v_l))|e|^2d\mu_g,
\end{aligns}
	where the equalities hold if and only if $\Delta
        e=\Delta\o{\mc{G}}e=0$. Now we claim that 
	\begin{align}\label{2.7} g^{ij}\phi_{v\b{v}}\o{u^v_i}\o{u^v_j}g^{kl}\phi_{v\b{v}}u^v_k u^v_l\leq (g^{ij}\phi_{v\b{v}}u^v_i\o{u^v_j})^2,\end{align}
and the equality holds if and only if $u^v_i=c u^v_j$. In fact, by taking normal coordinates at a fixed point, $g_{ij}=\delta_{ij}$, the above inequality is equivalent to 
	\begin{align*}\sum_{i<j} \left(\text{Re}((\o{u^v_i})^2 (u^v_j)^2)-|u^v_i|^2|u^v_j|^2\right)\leq 0.\end{align*}
	Denote $u^v_i=a+bi$, $u^v_j=c+di$, then 
	\begin{align*}|u^v_i|^2|u^v_j|^2-\text{Re}((\o{u^v_i})^2 (u^v_j)^2)=2(ad-bc)^2\geq 0,\end{align*}
	and the equality holds iff $u^v_i=c u^v_j$ for some constant
        $c$, which completes the proof of (\ref{2.7}). Substituting
        (\ref{2.7}) into (\ref{2.8}) gives
	\begin{align}\label{2.9}
		\langle (\mc{L}-{\mc{G}}\mc{L}^{-1}\o{\mc{G}}) e, e \rangle\geq \langle (\frac{1}{2}|du|^2-{\mc{G}}\mc{L}^{-1}\o{\mc{G}}) e, e \rangle \geq 0.
	\end{align}
Moreover, if $e\in\text{Ker}\left(\mc{L}-\mc{G}\mc{L}^{-1}\o{\mc{G}}\right)$, then the equality in  (\ref{2.9}) holds, which implies $e\in \text{Ker}\Delta$. The symmetricity of  $\mc{L}-\mc{G}\mc{L}^{-1}\o{\mc{G}}$  and $\frac{1}{2}|du|^2-\mc{G}\mc{L}^{-1}\o{\mc{G}}$ follows from 
\begin{align*}
	\langle \o{\mc{G}}(e_1), e_2\rangle=\langle e_1, \mc{G}(e_2)\rangle
\end{align*}
for any $e_1\in A^0(M, u^*T\mc{X}_z)$ and $e_2\in A^0(M, u^*\o{T\mc{X}_z})$.
\end{proof}
From Lemma \ref{lemma5}, we have
\begin{align}\label{2.11}
\left(\mc{L}-\mc{G}\mc{L}^{-1}\o{\mc{G}}\right)(W)=-\n^*A.
\end{align}
By taking inverse $\left(\mc{L}-\mc{G}\mc{L}^{-1}\o{\mc{G}}\right)^{-1}$ to both sides of (\ref{2.11}), 
$$W\equiv -\left(\mc{L}-\mc{G}\mc{L}^{-1}\o{\mc{G}}\right)^{-1}\n^*A \quad \text{mod}\quad\text{Ker}\left(\mc{L}-\mc{G}\mc{L}^{-1}\o{\mc{G}}\right)$$
Combining with (\ref{2.10}), we have 
\begin{align}\label{2.12}
\n W=-\n\left(\mc{L}-\mc{G}\mc{L}^{-1}\o{\mc{G}}\right)^{-1}\n^*A.	
\end{align}
Substituting (\ref{2.12}) into (\ref{2.3}), we obtain the second variation of 
the energy.
\begin{thm}\label{thm2.2}
	The second variation of the energy is as follows:
	\begin{align}\label{2.19}
	\frac{\p^2}{\p z\p\b{z}}E(z)=\frac{1}{2}\int_M c(\phi)_{z\b{z}}|du|^2d\mu_g+\langle(Id-\n\left(\mc{L}-\mc{G}\mc{L}^{-1}\o{\mc{G}}\right)^{-1}\n^*)A,A\rangle.	
	\end{align}
\end{thm}
\begin{proof}
From (\ref{2.3}) and (\ref{2.12}), we have 
\begin{align*}
	\frac{\p^2}{\p z\p\b{z}}E(z) &=\int_M \left(c(\phi)_{z\b{z}}g^{ij}\phi_{v\b{v}}u^v_i\o{u^v_j}+g^{ij}A_{z\b{v}\b{v}}A^{\b{v}}_{\b{z}v}u^v_i\o{u^v_j}-g^{ij}\n_i A_{z\b{v}\b{v}}\o{u^v_j}(\o{u^v_z}-\o{a^v_z})\right)d\mu_g\\
	&=\frac{1}{2}\int_M c(\phi)_{z\b{z}}|du|^2 d\mu_g+ \langle A, A\rangle+\langle\n^*A, W\rangle\\
	&=\frac{1}{2}\int_M c(\phi)_{z\b{z}}|du|^2 d\mu_g+ \langle A, A\rangle+\langle A, \n W\rangle\\
	&=\frac{1}{2}\int_M c(\phi)_{z\b{z}}|du|^2 d\mu_g+ \langle A, A\rangle+\langle A, -\n\left(\mc{L}-\mc{G}\mc{L}^{-1}\o{\mc{G}}\right)^{-1}\n^*A\rangle\\
	&=\frac{1}{2}\int_M c(\phi)_{z\b{z}}|du|^2d\mu_g+\langle(Id-\n\left(\mc{L}-\mc{G}\mc{L}^{-1}\o{\mc{G}}\right)^{-1}\n^*)A,A\rangle,
\end{align*}
where the last equality follows from Lemma \ref{lemma6}, and that $\mc{L}-\mc{G}\mc{L}^{-1}\o{\mc{G}}$ is symmetric. 
\end{proof}

\begin{prop}\label{prop1}
If $\dim M=1$, then 	
\begin{align*}
	\frac{\p^2 E^{1/2}}{\p z\p\b{z}} 
 &=	\frac{1}{2}\frac{1}{E^{1/2}}\left(\int_M (\Box+1)^{-1}(|A|^2)d\mu_g+\langle\frac{1}{2}|du|^2(|du|^2+\Delta)^{-1}A, A\rangle\right),
\end{align*}
where $\Box=-\phi_{v\b{v}}\p_v\p_{\b{v}}$ and $|A|^2=|A_{z\b{v}}^v|^2(\frac{1}{2}|du|^2)$ is a smooth function on $(z, v)=(z,u(z,x))$.
If we take the arc-length parametrization at $z=z_0$, i.e. $\frac{1}{2}|du|^2(z_0)=1$, then the first and the second variations  of geodesic length function are given by
\begin{align*}
	\frac{\p\ell(z)}{\p z}|_{z=z_0}=\frac{1}{2}\langle A,du\rangle
\end{align*}
and
 \begin{align*}
 \frac{\p^2\ell(z)}{\p z\p\b{z}}|_{z=z_0}=\frac{1}{2}\left(\int_M (\Box+1)^{-1}(|A|^2)d\mu_g+\langle(2+\Delta)^{-1}A, A\rangle\right).
 \end{align*}
\end{prop}
\begin{proof}
	By the condition $\dim M=1$, we denote $g=g_{tt}dt\otimes dt$, then the harmonic equation (\ref{harmonic}) is reduced to
	\begin{align}\label{harmonicdim1}
	\n_t u^v_t=\p_t u^v_t-\Gamma^t_{tt}u^v_t+\phi_v (u^v_t)^2=0,	
	\end{align}
        where $\Gamma^t_{tt}=\frac{1}{2}\p_t\log g_{tt}$.
        It gives then
\begin{align}
\n_t(\frac{1}{2}|du|^2)=\n_t(g^{tt}\phi_{v\b{v}}u^v_t\o{u^v_t})=g^{tt}\phi_{v\b{v}}(\n_t u^v_t\o{u^v_t}+u^v_t\o{\n_t u^v_t})=0,	
\end{align}
which implies that $|du|^2$ is a constant on $M$ for each $z$. Also by (\ref{harmonicdim1}), one has
\begin{align*}
(g^{tt}\phi_{v\b{v}}u^v_t u^v_t \o{e^v})_{;tt}=	g^{tt}\phi_{v\b{v}}u^v_t u^v_t \o{e^v}_{;tt},
\end{align*}
which concludes that $\mc{L}\mc{G}=\mc{G}\mc{L}$ when acting on the element in $A^0(M, u^*\o{T\mc{X}_z})$, thus
\begin{align*}
\mc{G}\mc{L}^{-1}\b{\mc{G}}=\mc{L}^{-1}(\mc{L}\mc{G}-\mc{G}\mc{L})\mc{L}^{-1}\b{\mc{G}}+\mc{L}^{-1}\mc{G}\b{\mc{G}}=\mc{L}^{-1}(\frac{1}{2}|du|^2)^2,	
\end{align*}
where the last equality follows from 
$$\mc{G}\b{\mc{G}}=(g^{tt}\phi_{v\b{v}}u^v_tu^v_t)(g^{tt}\phi_{v\b{v}}\o{u^v_t}\o{u^v_t})=(g^{tt}\phi_{v\b{v}}u^v_t\o{u^v_t})^2=(\frac{1}{2}|du|^2)^2.$$
In $\dim M=1$, then $\n^2=0$, and
\begin{align}
\n\Delta=\n(\n\n^*+\n^*\n)=\n\n^*\n=\Delta\n,	
\end{align}
which implies that $\n \mc{L}=\mc{L}\n$ by $\mc{L}=\Delta+\frac{1}{2}|du|^2$ and noting that $|du|^2$ is constant. Thus
\begin{aligns}
\n\left(\mc{L}-\mc{G}\mc{L}^{-1}\o{\mc{G}}\right)^{-1}\n^*A	&=\n\left(\mc{L}-\mc{L}^{-1}(\frac{1}{2}|du|^2)^2\right)^{-1}\n^*A\\
&=\left(\mc{L}-\mc{L}^{-1}(\frac{1}{2}|du|^2)^2\right)^{-1}\n\n^*A\\
&=\left(\mc{L}-\mc{L}^{-1}(\frac{1}{2}|du|^2)^2\right)^{-1}\Delta A
\end{aligns}
by noting $\n A=0$ (see Lemma \ref{lemma5} (i)).
Further the eigenvector decomposition method
of \cite[Lemma 7.2]{AS} implies that the last term is
 \begin{align}\label{2.18}
 \left(\mc{L}-\mc{L}^{-1}(\frac{1}{2}|du|^2)^2\right)^{-1}\Delta A=A-\frac{1}{2}|du|^2(|du|^2+\Delta)^{-1}A-\frac{1}{2}\mb{H}(A)
 \end{align}
 
We substitute now (\ref{2.18}) into (\ref{2.19}), and use Lemma
\ref{lemma1} (v),
to find
\begin{aligns}\label{2.20}
  &\quad
  \frac{\p^2}{\p z\p\b{z}}E(z)
  \\
  &=\frac{1}{2}\int_M (\Box+1)^{-1}(\frac{|A|^2}{\frac{1}{2}|du|^2})|du|^2d\mu_g+\langle\frac{1}{2}|du|^2(|du|^2+\Delta)^{-1}A+\frac{1}{2}\mb{H}(A), A\rangle\\
	&=\int_M (\Box+1)^{-1}(|A|^2)d\mu_g+\langle\frac{1}{2}|du|^2(|du|^2+\Delta)^{-1}A, A\rangle+\frac{1}{2}\|\mb{H}(A) \|^2.
	\end{aligns}
By  Proposition \ref{prop1},  $u^v_tdt\otimes \frac{\p}{\p v}\in A^1(M,u^*T\mc{X}_z)$ is harmonic, which is unique up to a constant factor  since $\dim M=1$. Thus 
\begin{align}
\|\mb{H}(A)	\|^2=\left|\frac{1}{\|u^v_tdt\otimes \frac{\p}{\p v}\|}\langle A,u^v_tdt\otimes \frac{\p}{\p v}\rangle\right|^2 =\frac{1}{E}|\langle A,du\rangle|^2=\frac{1}{E}\left|\frac{\p E}{\p z}\right|^2,
\end{align}
where the last equality follows from Theorem \ref{thm2.1}.
The equality (\ref{2.20}) now becomes
 \begin{aligns}\label{squareroot}
 \frac{\p^2 E^{1/2}}{\p z\p\b{z}} &=\frac{1}{2}E^{-1/2}\left(\frac{\p^2}{\p z\p\b{z}}E-\frac{1}{2E}\left|\frac{\p E}{\p z}\right|^2\right)\\
 &=	\frac{1}{2}\frac{1}{E^{1/2}}\left(\int_M (\Box+1)^{-1}(|A|^2)d\mu_g+\langle\frac{1}{2}|du|^2(|du|^2+\Delta)^{-1}A, A\rangle\right).
 \end{aligns}
If we take the arc-length parametrization at $z=z_0$, i.e. $\frac{1}{2}|du|^2=1$ at $z=z_0$, denote $\ell_0:=\int_M d\mu_g$, then the geodesic length function is 
\begin{aligns}\label{Length-Energy}
\ell(z)&:=\int_M \sqrt{g^{tt}\phi_{v\b{v}}u^v_t\o{u^v_t}}d\mu_g=\int_M (\frac{1}{\sqrt{2}}|du|)d\mu_g\\
&=	\frac{1}{\sqrt{2}}|du| \ell_0=\left(\int_M \frac{1}{2}|du|^2d\mu_g\right)^{1/2}\ell_0^{1/2}\\
&=E^{1/2}\ell_0^{1/2},
\end{aligns}
 and $\ell(z_0)=\ell_0$. From  Theorem \ref{thm2.1} and (\ref{squareroot}), we obtain  the first and the second variations  of geodesic length function 
\begin{align}
  \label{apple}
  \frac{\p\ell(z)}{\p z}|_{z=z_0}=\left(\frac{1}{2}E^{-1/2}\ell_0^{1/2}\frac{\p E}{\p z}\right)|_{z=z_0}=\left(\frac{1}{2}\frac{\p E}{\p z}\right)|_{z=z_0}=\frac{1}{2}\langle A,du\rangle
\end{align}
and
 \begin{align}\label{S2length}
 \frac{\p^2\ell(z)}{\p z\p\b{z}}|_{z=z_0}=\frac{1}{2}\left(\int_M (\Box+1)^{-1}(|A|^2)d\mu_g+\langle(2+\Delta)^{-1}A, A\rangle\right).
 \end{align}
 \end{proof} 
\begin{rem}\label{rem}
  Note that the above formula
  (\ref{apple}), in the special case of Euclidean metric
  on the circle with $g^{tt}=1$, and $\frac{1}{2}|du|^2=
  g^{tt}u_t^v\overline{u_t^v}\phi_{v\bar v}=
  u_t^v\overline{u_t^v}\phi_{v\bar v}=1$ at $z_0$, takes the following
  form
  $$\frac{\p\ell(z)}{\p z}|_{z=z_0}=\frac{1}{2}\langle
  A,du\rangle=\frac{1}{2}\int A_{z\bar v\bar
    v}\overline{u_t^vu_t^v}g^{tt} dt=\frac{1}{2}\int A_{z\bar v\bar
    v}\overline{u_t^vu_t^v} dt=\frac{1}{2}\int_{\gamma_z} A_z.$$
  This agrees with the one given in \cite[Theorem 1.1]{AS}, where the last equality follows from \cite[Definition 3.2]{AS}.
 However  comparing (\ref{S2length}) with \cite[Theorem 6.2,  (38)]{AS}, we
 find there is a
 extra term
 $\frac{1}{4\ell(\gamma_s)}\int_{\gamma_s}A_i\cdot\int_{\gamma_s}A_{\b{j}}$
 in their formula. This minor error is due the following: from (\ref{Length-Energy}), the first variation is 
 \begin{align}
 \frac{\p\ell}{\p z}=\frac{1}{2}E^{-1/2}\ell_0^{1/2}\frac{\p E}{\p z},	
 \end{align}
and the second variation has two terms 
\begin{aligns}
 \frac{\p^2\ell(z)}{\p z\p\b{z}}|_{z=z_0}&=\left(\frac{1}{2}E^{-1/2}\ell_0^{1/2}\frac{\p^2 E}{\p z\p\b{z}}-\frac{1}{4}E^{-3/2}\ell_0^{1/2}\frac{\p E}{\p z}\frac{\p E}{\p \b{z}}\right)|_{z=z_0}\\
 &=\frac{1}{2}\frac{\p^2 E}{\p z\p\b{z}}-\frac{1}{4\ell_0}\frac{\p E}{\p z}\frac{\p E}{\p \b{z}}.
\end{aligns}
 So the term $-\frac{1}{4\ell_0}\frac{\p E}{\p z}\frac{\p E}{\p \b{z}}$ was lost in their computations. 
\end{rem}

\subsection{Plurisubharmonicity}\label{subsec2.3}
In this section, we will prove the logarithm of the energy $\log E(z)$ is strictly plurisubharmonic on Teichm\"uller space. 

\begin{lemma}\label{lemma7}
The operator 
\begin{align}
  \n \Delta^{-1} \n^*-\n(\mc{L}-{\mc{G}}\mc{L}^{-1}\o{\mc{G}})^{-1}\n^*
\end{align}
is
 non-negative when acting on $
 A^1(M, u^*T\mc{X}_z)$, i.e,
 $$
 \langle (\n \Delta^{-1} \n^*-\n(\mc{L}-{\mc{G}}\mc{L}^{-1}\o{\mc{G}})^{-1}\n^* )f, f
 \rangle\ge 0
 $$
 for any $f\in
 A^1(M, u^*T\mc{X}_z)$.
\end{lemma}
 \begin{proof}
 For any $f\in A^1(M, u^*T\mc{X}_z)$, we denote $e=\n^* f\in A^0(M, u^*T\mc{X}_z)$.
 Denote $D_1=\mc{L}-\mc{G}\mc{L}^{-1}\o{\mc{G}}$ and $D_2=\Delta$.
 So
  $D_1$ is non-negative and symmetric, and
\begin{align}\label{2.15}
\langle (D_1-D_2) \tilde{e}, \tilde{e}\rangle=\left\langle\left(\frac{1}{2}|du|^2-{\mc{G}}\mc{L}^{-1}\o{\mc{G}}\right)\tilde{e}, \tilde{e}\right\rangle \geq 0
\end{align}
for any $\tilde{e}\in A^0(M, u^*T\mc{X}_z)$,
by Lemma \ref{lemma6}, 
Since
	\begin{align*}
	D_2^{-1}-D_1^{-1}=D_2^{-1}(D_1-D_2)D_1^{-1},	
	\end{align*}
so
\begin{aligns}\label{2.16}
 \langle (D_2^{-1}-D_1^{-1})e,e\rangle &=\langle D_2^{-1}(D_1-D_2)D_1^{-1}e,e\rangle\\
 &=\langle (D_1-D_2)D_1^{-1}e, D_2^{-1}e\rangle\\
 &=\langle (D_1-D_2)D_1^{-1}e,(D_1^{-1}+D_2^{-1}(D_1-D_2)D_1^{-1})e\rangle\\
 &=\langle(D_1-D_2)D_1^{-1}e, D_1^{-1}e\rangle\\
 &\quad+\langle D_2^{-1}(D_1-D_2)D_1^{-1}e,(D_1-D_2)D_1^{-1}e\rangle\\
 &\geq 0,
\end{aligns}
where the last inequality holds by (\ref{2.15}). 
From (\ref{2.16}), we get
\begin{align*}
	\langle(\n \Delta^{-1} \n^*-\n(\mc{L}-{\mc{G}}\mc{L}^{-1}\o{\mc{G}})^{-1}\n^*)f,f\rangle&=\langle( \Delta^{-1}-(\mc{L}-{\mc{G}}\mc{L}^{-1}\o{\mc{G}})^{-1})\n^*f,\n^*f\rangle\\
	&=\langle( \Delta^{-1}-(\mc{L}-{\mc{G}}\mc{L}^{-1}\o{\mc{G}})^{-1})e,e\rangle\geq 0. 
\end{align*}
 \end{proof}
From Lemma \ref{lemma8}, \ref{lemma9}, \ref{lemma7} and Theorem
\ref{thm2.2},
we conclude that
\begin{aligns}\label{2.13}
	\frac{\p^2}{\p z\p\b{z}}E(z)&=\frac{1}{2}\int_M c(\phi)_{z\b{z}}|du|^2d\mu_g+\langle(Id-\n\left(\mc{L}-\mc{G}\mc{L}^{-1}\o{\mc{G}}\right)^{-1}\n^*)A,A\rangle\\
	&=\frac{1}{2}\int_M c(\phi)_{z\b{z}}|du|^2d\mu_g+\langle(\Delta^{-1}\Delta-\n\Delta^{-1}\n^*)A,A\rangle\\
	&\quad+\langle(\n\Delta^{-1}\n^*-\n\left(\mc{L}-\mc{G}\mc{L}^{-1}\o{\mc{G}}\right)^{-1}\n^*)A,A\rangle+\langle\mb{H}(A),A\rangle\\
	&\geq \frac{1}{2}\int_M c(\phi)_{z\b{z}}|du|^2d\mu_g+\|\mb{H}(A)\|^2.
\end{aligns}
 Note that $u^v_ldx^l\otimes\frac{\p}{\p v}$ is  harmonic (see Proposition \ref{prop1}), so 
\begin{aligns}\label{2.14}
\|\mb{H}(A)\|^2&\geq \frac{1}{\|u^v_ldx^l\otimes \frac{\p}{\p v}\|^2}|\langle A, u^v_ldx^l\otimes \frac{\p}{\p v}\rangle|^2\\
&=\frac{1}{E}\left|\int_M A_{z\b{v}\b{v}}\b{u}^v_i \b{u}^v_j g^{ij}d\mu_g\right|^2=	\frac{1}{E}\frac{\p E}{\p z}\frac{\p E}{\p\b{z}},
\end{aligns}
where the last equality follows from Theorem \ref{thm2.1}.
Substituting (\ref{2.14}) into (\ref{2.13}), we obtain
\begin{align*}
\frac{\p^2}{\p z\p\b{z}}\log E(z) &=-\frac{1}{E^2}\frac{\p E}{\p z}	\frac{\p E}{\p\b{z}}+\frac{1}{E}\frac{\p^2 E}{\p z\p\b{z}}\\
&\geq -\frac{1}{E^2}\frac{\p E}{\p z}	\frac{\p E}{\p\b{z}}+\frac{1}{E}\left(\frac{1}{E}\frac{\p E}{\p z}\frac{\p E}{\p\b{z}}+\frac{1}{2}\int_M c(\phi)_{z\b{z}}|du|^2d\mu_g\right)\\
&= \frac{1}{\|du\|^2}\int_M c(\phi)_{z\b{z}}|du|^2d\mu_g>0
\end{align*}
by Lemma \ref{lemma1} (vi) and noting that Teichm\"uller curve $\pi:\mc{X}\to \mc{T}$ is not infinitesimally trivial.
\begin{thm}\label{thm2.3}
Let $\pi:\mc{X}\to \mc{T}$ be Teichm\"uller curve over Teichm\"uller space $\mc{T}$. Let $(M^n, g)$ be a Riemannian manifold and consider the energy of the harmonic map from $(M^n,g)$ to $\mc{X}_z$,  $z\in \mc{T}$. 
Then the logarithm of energy $\log E(z)$ is a strictly plurisubharmonic function on Teichm\"uller space. In particular, the energy function is also strictly plurisubharmonic. 	
\end{thm}

By \cite[Lemma 3]{Sch1}, for any two positive functions $a, b$, one has 
\begin{align*}
(a+b)\sqrt{-1}\p\b{\p}\log (a+b)\geq a\sqrt{-1}\p\b{\p}\log a+b\sqrt{-1}\p\b{\p}\log b.
\end{align*}
Combining with the above inequality we have 
\begin{cor}
	The logarithm of a sum of the energy functions 
	$$\log \sum_{i=1}^{N}E_i(z)$$
	is also strictly plurisubharmonic. 
\end{cor}

As a corollary, we proved
\begin{cor}[\cite{Wol1,Wol2, Wol3}]\label{geodesic}
Let $\gamma(z)$ be a smooth family of closed geodesic curves over  Teichm\"uller space. Then both the length function $\ell(\gamma(z))$ and the logarithm of length function $\log \ell(\gamma(z))$ are strictly plurisubharmonic.	In particular, the geodesic length function is strictly convex along Weil-Petersson geodesics.
\end{cor}
\begin{proof}
	From (\ref{Length-Energy}), the relation between the geodesic length function and the energy function is 
\begin{align}\label{LE}
\ell(\gamma(z))=E(z)^{1/2}\ell^{1/2}_0.	
\end{align}
From Theorem \ref{thm2.3}, one concludes that $\log \ell(\gamma(z))$ is  strictly plurisubharmonic, which implies that $\ell(\gamma(z))$ is also a strict plurisubharmonic function.
The strict convexity of geodesic length function along Weil-Petersson geodesics follows from the following comparison between the complex Hessian and WP Riemannian Hessian (see \cite[Section 3]{Wol2})
$$\p\b{\p}\ell\leq \ddot{\ell}\leq 3\p\b{\p}\ell.$$
\end{proof}
In the next section we shall prove a general convexity result
along Weil-Petterson geodesics for general harmonic maps $u: M\to
\mathcal X_z$ instead of a closed geodesic $u: S^1\to \mathcal X_z$.
\begin{defn}
A complex manifold $N$ is Stein if it admits a plurisubharmonic exhaustion (proper) function	$\mc{F}:N\to \mb{R}$.
\end{defn}

\begin{cor}
If $(u_0)_*: \pi_1(M)\to \pi_1(\mc{X}_{z_0})$ is surjective, then the energy function $E(z)$ is proper and strictly plurisubharmonic. In particular, Teichm\"uller space $\mc{T}$ is a complex Stein manifold. 
\end{cor}
\begin{proof}
The first part follows from \cite[Proposition 3.1.1]{Yamada}. For the second part, we take $M=\mc{X}_{z_0}$ and $u_0=Id$, then $(u_0)_*: \pi_1(M)\to \pi_1(\mc{X}_{z_0})$ is surjective. In this case, the energy function is  proper and strictly plurisubharmonic, which implies that Teichm\"uller space is Stein. 	
\end{proof}

\section{Convexity of energy along Weil-Petersson geodesic}\label{sec3}

In this section, we will give a simple proof on the convexity of the energy along Weil-Petersson geodesics   \cite[Theorem 3.1.1]{Yamada}.

Let $\mc{M}_{-1}$ denote the space of hyperbolic metrics. Now suppose that $\sigma(t)$ is a Weil-Petersson geodesic parametrized by arc-length in Teichm\"uller space $\mc{T}=\mc{M}_{-1}/\mc{D}_0$, where $\mc{D}_0$ is the identity component  of the diffeomorphism group. Then we can lift $\sigma(t)$ horizontally to $\mc{M}_{-1}$. The lift $\Phi_t$ is itself a geodesic in $\mc{M}_{-1}$ with its tangent vector $h$ in $T_{\Phi_t}\mc{M}_{-1}$ satisfying the tracefree, transverse  condition (see e.g. \cite{Ficher} or \cite[(1)]{Yamada}),
\begin{align}
\text{Tr}_{\Phi}h=0\quad \delta_{\Phi}h=0.	
\end{align}
From \cite[(3.4)]{Wolf}, the metrics $\Phi_t$ satisfy 
\begin{multline}\label{WPG}
	\Phi_t=\phi_0dvd\b{v}+t(q dv^2+\o{q}d\b{v}^2)\\
	+t^2/2\left(\frac{2|q|^2}{\phi^2_0}-2(\Delta-2)^{-1}\frac{2|q|^2}{\phi^2_0}\right)\phi_0dvd\b{v}+O(t^4).
\end{multline}
 Here $qdv^2$ is a holomorphic quadratic form, $\phi_0dvd\b{v}$ is a
 hyperbolic metric. We denote by $\Phi$ the  matrix representation
 of $\Phi_t$ with respect to the basis $\{dv, d\b{v}\}$, i.e. 
 \begin{align}
 \Phi_t=(dv, d\b{v})\Phi\otimes\left(\begin{matrix}{}
 dv \\
d\b{v}	
\end{matrix}\right). 	
 \end{align}
Then 
\begin{align}
\Phi=\left(\begin{matrix}{}
 \Phi_{vv} & \Phi_{v\b{v}}\\
 \Phi_{v\b{v}} & \Phi_{\b{v}\b{v}}	
\end{matrix}
\right)	=\left(\begin{matrix}{}
 tq  & \frac{\phi_0}{2}+\frac{t^2}{2}\left(\frac{|q|^2}{\phi^2_0}+\alpha\right)\phi_0\\
 \frac{\phi_0}{2}+\frac{t^2}{2}\left(\frac{|q|^2}{\phi^2_0}+\alpha\right)\phi_0 & t\o{q}   
\end{matrix}
\right)+O(t^4),
\end{align}
where
\begin{align}
  \label{3.5}
\alpha=-(\Delta-2)^{-1}\frac{2|q|^2}{\phi^2_0}\geq\frac{1}{3}\frac{|q|^2}{\phi_0^2}> 0\quad a.e. ,\end{align} (see \cite[Lemma 5.1]{Wolf}). 

Let $(M^n, g)$ be a Riemannian manifold and consider the energy $E(u)$ of a smooth map $u$ from $(M^n, g)\to (\Sigma, \Phi_t)$. Let $\tilde{u}: M\to \Sigma$ be a fixed smooth map.  By Theorem \ref{1.1}, for each $t$, there exists a harmonic map $u(t)$ homotopic to $\tilde{u}$ and is unique unless its image is a point or a geodesic. Following the argument in \cite[Page 36]{Yamada}, the following function
\begin{align}
E(t):=E(u(t))	
\end{align}
is well-defined and smooth.  Note that the metric $\Phi_t\in A^0(\Sigma,\otimes^2 T^*\Sigma)$, so $u^*\Phi_t\in A^0(M,\otimes^2T^*M)$, and is given by
\begin{aligns}
u^*\Phi_t &=u^*\left(\Phi_{vv}dv\otimes dv+\Phi_{v\b{v}}dv\otimes d\b{v}+\Phi_{v\b{v}}d\b{v}\otimes dv+\Phi_{\b{v}\b{v}}d\b{v}\otimes d\b{v}\right)	\\
&=\left(\Phi_{vv}u^v_i u^v_j+\Phi_{v\b{v}}u^v_i\o{u^v_j}+\Phi_{v\b{v}}\o{u^v_i}u^v_j+\Phi_{v\b{v}}\o{u^v_i}\o{u^v_j}\right)dx^i\otimes dx^j\\
&=\Phi_{\alpha\beta}u^\alpha_{i}u^{\beta}_j dx^i\otimes dx^j,
\end{aligns}
where $\alpha,\beta\in \{v,\b{v}\}$ and $u^{\b{v}}_i:=\o{u^v_i}$. 
Recall  the trace
$\text{Tr}_g$ with respect to the Riemannian metric $g$.
Then the energy $E(t)$ can be expressed as 
\begin{align}\label{energy1}
E(t)=E(u(t))=\frac{1}{2}\int_M g^{ij}u^{\alpha}_i u^{\beta}_j \Phi_{\alpha\beta}d\mu_g=\frac{1}{2}\int_M \text{Tr}_g (u(t)^*\Phi_t)d\mu_g.	
\end{align}

\begin{thm}[{\cite[Theorem 3.1.1]{Yamada}}]\label{Yamada}
 Under the assumptions above, the function $E(t)$ is a strictly convex function in $t$, and hence the energy function $E: \mc{T}\to \mb{R}$ is strictly convex along any Weil-Petersson geodesic in $\mc{T}$.	
\end{thm}
\begin{proof}
  The metrics $\Phi_t$
in    (\ref{WPG}) and their first and second derivatives
at $t=0$ is
  \begin{align}\label{3.1}
  \Phi_0
  =\phi_0dvd\b{v}=\frac{\phi_0}{2}(dv\otimes d\b{v}+d\b{v}\otimes dv),\quad\dot{\Phi}_0=q dv^2+\o{q dv^2},\quad
 \ddot{\Phi}_0=\left(\frac{2|q|^2}{\phi^2_0}+2\alpha\right)\Phi_0. 	
\end{align}
By (\ref{energy1}), the energy at $t=0$ is 
\begin{align}\label{energy0}
E(0)=\frac{1}{2}\int_M g^{ij}u^v_i\o{u^v_j}\phi_0 d\mu_g. 	
\end{align}
	From \cite[Page 58, lemma 3.1.1]{Yamada}, the second derivative of $E(t)$ is given by
	\begin{align}\label{3.2}
	\frac{d^2}{dt^2}|_{t=0}E(t)&=\frac{1}{2}\int_M \text{Tr}_g (u^*_0\ddot{\Phi}_0)d\mu_g-\delta^2 E(u_0)(W_0,W_0),
	\end{align}
where $\text{Tr}_g (u^*_0\ddot{\Phi}_0):=g^{i\b{j}}(u_0)_{i}^{\alpha}(u_0)_j^{\beta}(\ddot{\Phi}_0)_{\alpha\b{\beta}}$, $\alpha,\beta\in \{v,\b{v}\}$, $W_0=\frac{d}{dt}|_{t=0}u(t)$ and 
\begin{align}\label{3.3}\delta^2E(u_0)(W_0,W_0)=-\frac{1}{2}\int_M (\dot{\Phi}_0)_{\alpha\beta}(\n W_0)^{\alpha}u^{\beta}_j g^{ij}d\mu_g.\end{align}
We substitute (\ref{3.1}) into (\ref{3.3}) and estimate
is from above Cauchy-Schwarz inequality,
\begin{aligns}\label{3.4}
	\delta^2E(u_0)(W_0,W_0)&=-\frac{1}{2}\int_M (\dot{\Phi}_0)_{\alpha\beta}(\n W_0)^{\alpha}u^{\beta}_j g^{ij}d\mu_g\\
	&=-\frac{1}{2}\int_M (q(\n W_0)^v u^v_jg^{ij}+\o{q(\n W_0)^vu^v_jg^{ij}})d\mu_g\\
&=-\text{Re}\int_M \left((\n_i W_0)^vqu^v_jg^{ij}\right)d\mu_g\\
&\leq \int_M\left (\frac{1}{2}g^{ij}(\n_iW_0)^v\o{(\n_jW_0)^v}\phi_0+\frac{1}{2}g^{ij}u^v_i\o{u^v_j}\frac{|q|^2}{\phi_0}\right)d\mu_g\\
&=\frac{1}{2}\int_M\left(g^{ij}(\n_iW_0)^v\o{(\n_jW_0)^{v}}\frac{\phi_0}{2}+g^{ij}\o{(\n_jW_0)^v}(\n_iW_0)^v\frac{\phi_0}{2}\right)d\mu_g\\
&\quad+\frac{1}{2}\int_M\left(g^{ij}u^v_i\o{u^v_j}\frac{\phi_0}{2}+g^{ij}\o{u^v_j}u^v_i\frac{\phi_0}{2}\right)\frac{|q|^2}{\phi_0^2}d\mu_g\\
&= \frac{1}{2}\|\n W_0\|^2+\frac{1}{2}\int_M\frac{|q|^2}{\phi_0^2}
\text{Tr}_{g}(u_0^*\Phi_0)d\mu_g\\
&\leq \frac{1}{2}\delta^2E(u_0)(W_0,W_0)+\frac{1}{2}\int_M\frac{|q|^2}{\phi_0^2}\text{Tr}_{g}(u_0^*\Phi_0)d\mu_g,
\end{aligns}
where the last inequality follows from  \cite[Page 15]{Jost} or \cite[Page 62]{Yamada}, $\|\n W_0\|^2\leq \delta^2E(u_0)(W_0,W_0)$. From (\ref{3.4}), we conclude that 
\begin{align}
	\delta^2E(u_0)(W_0,W_0)\leq \int_M\frac{|q|^2}{\phi_0^2}\text{Tr}_{g}(u_0^*\Phi_0)d\mu_g.
\end{align}
Substituting (\ref{3.1}) and (\ref{3.4}) into (\ref{3.2}) and using (\ref{3.5}), we get
 \begin{aligns}\label{3.8}
\frac{d^2}{dt^2}|_{t=0}E(t)&=\frac{1}{2}\int_M \text{Tr}_g (u^*_0\ddot{\Phi}_0)d\mu_g-\delta^2 E(u_0)(W_0,W_0)\\
&\geq \int_M \left(\frac{|q|^2}{\phi^2_0}+\alpha\right)\text{Tr}_g(u_0^*\Phi_0)d\mu_g- \int_M\frac{|q|^2}{\phi_0^2}\text{Tr}_{g}(u_0^*\Phi_0)d\mu_g\\
&=\int_M \alpha \text{Tr}_g(u_0^*\Phi_0)d\mu_g\\
&\geq \int_M \frac{|q|^2}{3\phi_0^2}\text{Tr}_g(u_0^*\Phi_0)d\mu_g> 0,
\end{aligns}
which completes the proof.
\end{proof}

As a corollary,  we prove 
\begin{cor}
The function $E(t)^{c}, c>5/6$ (resp. $c=5/6$) is strictly convex (resp. convex) along a Weil-Petersson geodesic. 	
\end{cor}
\begin{proof} The second derivative 
	\begin{align}\label{3.5-1}
\frac{d^2}{dt^2}E(t)^c=cE^{c-2}\left(\left(c-1\right)\left(\frac{dE}{dt}\right)^2+E\frac{d^2E}{dt^2}\right). 
	\end{align}
        If $c\geq 1$ then this
gives
\begin{align}\label{3.9}
	\frac{d^2}{dt^2}|_{t=0}E(t)^c\geq cE^{c-1}\frac{d^2E}{dt^2}|_{t=0}>0
\end{align}
        by Theorem \ref{Yamada}.   
Now we assume that $5/6\leq c<1$. From \cite[Page 57]{Yamada}, the first derivative of the energy is 
\begin{align}\label{3.6}
\frac{dE}{dt}|_{t=0}=\frac{1}{2}\int_M \text{Tr}_g \left(u_0^*\dot{\Phi}_0\right)	d\mu_g=\text{Re}\int_M g^{ij}u^v_i u^v_j q d\mu_g.
\end{align}
The following quadratic polynomial in $q$ is non-negative
 \begin{aligns}
      \int_M (u^v_i-\o{u^v_i}\o{q}\phi_0^{-1}\lambda)(\o{u^v_j}-u^v_j q\phi_0^{-1}\o{\lambda})g^{ij}\phi_0 d\mu_g	
      \ge
    0
  \end{aligns}
where $\lambda=\int_M g^{ij}u^v_i u^v_j qd\mu_g/\int_M
g^{ij}u^v_i\o{u^v_j}|q|^2(\phi_0)^{-1}d\mu_g$ and $t=0$
in $u=u(t)$. Thus
 \begin{aligns}
\left| \int_M g^{ij}u^v_i u^v_j q d\mu_g\right|^2 &\leq \int_M g^{ij}u^v_i\o{u^v_j}\phi_0d\mu_g\int_M g^{ij}u^v_i\o{u^v_j}\frac{|q|^2}{\phi_0}d\mu_g\\
&=2E\int_M\frac{|q|^2}{\phi_0^2}\text{Tr}_{g}(u_0^*\Phi_0)d\mu_g \\
&\leq 6E\frac{d^2E}{dt^2},
 \end{aligns}
 where the second equality holds by (\ref{energy0}), the last inequality follows from (\ref{3.8}).
Combining with (\ref{3.6}) shows that 
\begin{align}\label{3.7}
\left(\frac{dE}{dt}\right)^2\leq 6E\frac{d^2E}{dt^2}.
\end{align}
Substituting (\ref{3.7}) into (\ref{3.5}) and using Theorem \ref{Yamada}, we have 
\begin{aligns}\label{3.10}
  \frac{d^2}{dt^2}E(t)^c
  &=cE^{c-2}\left(\left(c-1\right)\left(\frac{dE}{dt}\right)^2+E\frac{d^2E}{dt^2}\right)\\
		 &\geq cE^{c-2}\left(\left(c-1\right)6E\frac{d^2E}{dt^2}+E\frac{d^2E}{dt^2}\right)\\
                 &=c(6c-5)E^{c-1}\frac{d^2 E}{dt^2},
                           \end{aligns}
               and $\frac{d^2}{dt^2}E(t)^c =0$
               for $ c=5/6$ and 
               $   \frac{d^2}{dt^2}E(t)^c              >0$
               {for} $c\in (5/6,1)$ at $t=0$, where the second inequality holds since $c-1<0$ and (\ref{3.7}).  Combining (\ref{3.9}) with (\ref{3.10}), we complete the proof. 
\end{proof}

Another corollary is a positive answer to the Nielsen realization problem, which was answered by Kerckhoff \cite{Kerockhoff} long time ago. We say that a system of curves fills up the surface if the complement of the system is a union of disks.
\begin{cor}[{\cite[Theorem 5]{Kerockhoff}}]Any finite subgroup $G$ of the mapping class group of a surface $\Sigma$ can be realized as a isometry subgroup
of some hyperbolic metric on $\Sigma$.
\end{cor}
\begin{proof}Take a collection $\gamma=\cup \gamma_i$ of curves which fill up $\Sigma$. Viewing $\gamma$ as a geodesic map
from the union of circles into $\mathcal X_z, z\in \mathcal T$, we can consider the energy function $E(\gamma(z))$ over $\mathcal T$. By (\ref{Length-Energy}), $E(\gamma_i(z))=\frac{\ell(\gamma_i(z))^2}{\ell(\gamma_i(0))}$ where $\ell(\gamma_i(0))$ is the geodesic length of $\gamma_i$ at some point in $\mathcal T$. Then the sum $E(\gamma(z))=\sum E(\gamma_i(z))$ is strictly convex along a Weil-Petersson geodesic, and proper on $\mathcal T$, since the geodesic length function is proper on $\mathcal T$ (Lemma 3.1 in \cite{Kerockhoff}). Hence $E(\gamma(z))$ has a unique minimum point. Now consider the filling family $G\gamma$. Since $G\gamma$ is $G$-invariant, $E(G\gamma(z))=\sum_{\alpha\in G\gamma} \frac{\ell(\alpha(z))^2}{\ell(\alpha(0))}$ is $G$-invariant.
Then this function has a unique minimum point $z_0$. This point should be invariant under $G$, i.e., $G$ acts as an isometry group on $z_0$.
\end{proof}

\end{document}